\documentclass[11pt]{article}

\usepackage[T1]{fontenc}
\usepackage[numbers]{natbib}
\usepackage{graphicx}
\usepackage{amssymb}
\usepackage{amsmath,amsthm,enumerate}

\usepackage[citecolor=blue,colorlinks=true,unicode=true]{hyperref} 

\usepackage[left=2cm,right=2cm,top=3.5cm,bottom=3.5cm]{geometry}
\usepackage{color}
\usepackage{colortbl}
\usepackage{mdframed}
\usepackage{caption}
\usepackage{esint}

\renewcommand{\footnoterule}{%
  \kern -3pt
  \hrule width \textwidth height 0.4pt
  \kern 2pt
}

\usepackage{tcolorbox}
\usepackage{extarrows}
\usepackage{import}
\usepackage{mathtools}
\usepackage{upgreek}
\usepackage{lmodern}
\usepackage{inputenc}
\usepackage{microtype}
\usepackage{float}
\usepackage[shortlabels]{enumitem}
\setitemize{leftmargin=*}
\usepackage{tikz}
\usetikzlibrary{shapes,arrows}
\usepackage{chngcntr}
\usepackage[linesnumbered,german]{algorithm2e}
\usepackage{csquotes}
\usepackage{array}
\usepackage{multirow}
\usepackage{multicol}
\usepackage{cite}
\usepackage{accents}
\usepackage{bm}
\usepackage{mathrsfs}
\usepackage{xfrac}
\usepackage{hyperref}
\usepackage{xcolor}
\hypersetup{linktoc=all, colorlinks=true, linkcolor=blue, citecolor=blue, pdfauthor={Mária Lukáčová-Medvid'ová; Andreas Schömer}}
\usepackage{stackengine}

\stackMath

\DeclareFontShape{OMX}{cmex}{m}{b}{<-> cmexb10}{}
\DeclareSymbolFont{boldlargesymbols}{OMX}{cmex}{m}{b}
\DeclareMathAccent{\bwidetilde}{\mathord}{boldlargesymbols}{"65}

\newtheoremstyle{break}
{14pt}{20pt}%
{}{}%
{\bfseries}{\vspace{0.5mm}}%
{\newline}{}%

\theoremstyle{break}

\theoremstyle{plain}
\newtheorem{theorem}{Theorem}[section]

\newtheorem{lemma}[theorem]{Lemma}
\newtheorem{corollary}[theorem]{Corollary}

\theoremstyle{definition}

\theoremstyle{remark}
\newtheorem{remark}[theorem]{Remark}

\catcode`\@=11 \@addtoreset{equation}{section}

\catcode`\@=12

\allowdisplaybreaks

\definecolor{Cgrey}{rgb}{0.85,0.85,0.85}
\definecolor{Cblue}{rgb}{0.50,0.85,0.85}
\definecolor{Cred}{rgb}{1,0,0}
\definecolor{fancy}{rgb}{0.10,0.85,0.10}

\newcommand\Cbox[2]{%
    \newbox\contentbox%
    \newbox\bkgdbox%
    \setbox\contentbox\hbox to \hsize{%
        \vtop{
            \kern\columnsep
            \hbox to \hsize{%
                \kern\columnsep%
                \advance\hsize by -2\columnsep%
                \setlength{\textwidth}{\hsize}%
                \vbox{
                    \parskip=\baselineskip
                    \parindent=0bp
                    #2
                }%
                \kern\columnsep%
            }%
            \kern\columnsep%
        }%
    }%
    \setbox\bkgdbox\vbox{
        \color{#1}
        \hrule width  \wd\contentbox %
               height \ht\contentbox %
               depth  \dp\contentbox
        \color{black}
    }%
    \wd\bkgdbox=0bp%
    \vbox{\hbox to \hsize{\box\bkgdbox\box\contentbox}}%
    \vskip\baselineskip%
}

\newcommand{\reals}{\mathbb{R}}
\newcommand{\naturals}{\mathbb{N}}
\newcommand{\BMO}{\textit{BMO}}

\newcommand{\softd}{d\hspace{-0.2mm}'}

\newcommand{\fatx}{\bm{x}}
\newcommand{\faty}{\bm{y}}
\newcommand{\brho}{(\varrho_0)_{\Omega}}

\newcommand{\tZ}{\mathbf{\tilde{\text{$Z\,$}}}\!}
\newcommand{\tmu}{\mathbf{\tilde{\text{$\mu\,$}}}\!}

\newcommand{\tlambda}{\mathbf{\tilde{\text{$\lambda\,$}}}\!}

\newcommand{\trho}{\mathbf{\tilde{\text{$\varrho\,$}}}\!}

\newcommand{\tfatu}{\mathbf{\tilde{\text{$\bm{u}$}}}}
\newcommand{\fatudot}{\dot{\fatu}}

\newcommand{\fatF}{\bm{F}}
\newcommand{\fatf}{\bm{f}}

\newcommand{\fatG}{\bm{G}}

\newcommand{\fatu}{\bm{u}}
\newcommand{\fatU}{\bm{U}}
\newcommand{\fatX}{\bm{X}}
\newcommand{\fatv}{\bm{v}}
\newcommand{\fatw}{\bm{w}}
\newcommand{\fatV}{\bm{V}}
\newcommand{\fata}{\bm{a}}
\newcommand{\fatb}{\bm{b}}
\newcommand{\fatn}{\bm{n}}
\newcommand{\fatphi}{\bm{\phi}}
\newcommand{\fatpsi}{\bm{\psi}}
\newcommand{\fateta}{\bm{\eta}}
\newcommand{\Tmax}{T_{\mathrm{max}}}
\newcommand{\matrixA}{\mathbb{A}}
\newcommand{\matrixB}{\mathbb{B}}
\newcommand{\matrixF}{\mathbb{F}}

\newcommand{\po}{\partial\Omega}

\newcommand{\divx}{\mathrm{div}_{\hspace{-0.2mm}\fatx}}
\newcommand{\gradx}{\nabla_{\hspace{-0.7mm}\fatx}}
\newcommand{\deltax}{\Delta_{\fatx}}
\newcommand{\pt}{\partial_t}
\newcommand{\ppi}{\partial_i}
\newcommand{\ppj}{\partial_j}
\newcommand{\ppk}{\partial_k}

\newcommand{\into}{\int_{\Omega}}

\newcommand{\inttinto}{\int_0^{\,t}\!\!\into}
\newcommand{\dt}{\mathrm{d}t}

\newcommand{\ddt}{\frac{\mathrm{d}}{\mathrm{d}t}}
\newcommand{\ds}{\mathrm{d}s}
\newcommand{\dx}{\mathrm{d}\fatx}
\newcommand{\dy}{\mathrm{d}\faty}

\newcommand{\dxds}{\dx\hspace{0.3mm}\ds}

\makeatletter
\renewenvironment{proof}[1][\proofname]{%
   \par\pushQED{\qed}\normalfont%
   \topsep6\p@\@plus6\p@\relax
   \trivlist\item[\hskip\labelsep\itshape\bfseries#1\@addpunct{.}]%
   \ignorespaces
}{%
   \popQED\endtrivlist\@endpefalse
}
\makeatother

\date{}

\begin{document}


\title{Conditional regularity for the compressible Navier-Stokes equations with potential temperature transport\thanks{This work has been funded by the Deutsche Forschungsgemeinschaft (DFG, German Research Foundation), Project number 233630050-TRR 146. M.L. gratefully acknowledges support of the Gutenberg Research College of University Mainz and the Mainz Institute of Multiscale Modeling. The authors wish to thank E. Feireisl (Prague) for fruitful discussions.}}

%

\author{M\' aria Luk\' a\v cov\' a\,-Medvi\softd ov\' a
\and Andreas Sch\"omer
}

\date{\today}

\maketitle

\bigskip

\centerline{Institute of Mathematics, Johannes Gutenberg-University Mainz}
\centerline{Staudingerweg 9, 55128 Mainz, Germany}
\centerline{lukacova@uni-mainz.de, anschoem@uni-mainz.de}

\begin{abstract}
    We study \textcolor{black}{conditional regularity for} the compressible Navier-Stokes equations with potential temperature transport in a bounded domain $\Omega\subset\reals^d$, $d\in\{2,3\}$, with \textcolor{black}{no-slip boundary conditions.} We \textcolor{black}{first} prove the existence and uniqueness of local-in-time strong solutions. \textcolor{black}{Further}, we prove a blow-up criterion for the strong solution in terms of $L^\infty$-norms for the density and the velocity.
\end{abstract}




\section{Introduction}\label{sec_introduction}
In this paper, we consider the compressible Navier-Stokes equations with potential temperature transport. This model is often used in meteorological applications; see \citep{Klein}, \citep{Klein1}, \citep{Lukacova_Wiebe_1}, \citep{Lukacova_Wiebe_2}. It reads
\begin{alignat}{2}
    \varrho_t + \divx(\varrho\fatu) &= 0 && \qquad\text{in $Q_T$,} \label{original_1} \\[2mm]
    (\varrho\fatu)_t + \divx(\varrho\fatu\otimes\fatu) -\divx(\mathbb{S}(\gradx\fatu)) + \gradx p(\varrho\theta) &= \varrho\fatb && \qquad\text{in $Q_T$,} \label{original_2} \\[2mm]
    (\varrho\theta)_t + \divx(\varrho\theta\fatu) &= 0 && \qquad\text{in $Q_T$,} \label{original_3} \\[2mm]
    (\fatu,\varrho,\theta)(0,\cdot) &= (\fatu_0,\varrho_0,\theta_0) && \qquad \text{in $\Omega$}, \label{original_4} \\[2mm]
    \fatu|_{[0,T]\times\po} &= \bm{0}. && \label{original_7}
\end{alignat}
Here, $T>0$ is a given time and $\Omega\subset\reals^d$, $d\in\{2,3\}$, is a bounded domain and $Q_T=(0,T)\times\Omega$. $\varrho$ denotes the \textit{fluid density}, $\fatu$ the \textit{fluid velocity}, $\theta$ the \textit{potential temperature of the fluid} and $\fatb$ an \textit{external body force}. The \textit{pressure} is given by
\begin{align}
    p(\varrho\theta) = a(\varrho\theta)^\gamma, \qquad a>0, \quad \gamma>1. \label{pressure}
\end{align}
The viscous stress tensor $\mathbb{S}=\mathbb{S}(\gradx\fatu)$ reads
\begin{align}
    \mathbb{S}(\gradx\fatu) = \mu\!\left(\gradx\fatu + (\gradx\fatu)^T\right) + \lambda\,\divx(\fatu)\mathbb{I}, \label{stress}
\end{align}
where the constant viscosity coefficients $\mu,\lambda$ satisfy
\begin{align}
    \mu>0, \quad d\lambda+2\mu \geq 0. \label{viscosity_coefficients}
\end{align}
Despite the importance of the model \eqref{original_1}--\eqref{viscosity_coefficients}, its rigorous analysis is not yet complete. In \citep{Zatorska}, global-in-time weak solutions to \eqref{original_1}--\eqref{viscosity_coefficients} are studied. Their existence is only proven under the assumption $\gamma\geq 9/5$; see \citep[Theorem 1 with $\mathcal{T}(s)=s^\gamma$]{Zatorska}. However, the physically relevant values for the adiabatic index $\gamma$ lie in the interval $(1,2]$ for $d=2$ and in the interval $(1,5/3]$ if $d=3$. Therefore, we have recently introduced the concept of \textit{dissipative measure-valued (DMV) solutions} for this model and proved the existence of such solutions for all $\gamma>1$; see \citep{LS_existence}. Moreover, we proved a DMV-strong uniqueness principle stating that if a strong solution and a DMV solution emanate from the same initial data, then they coincide; see \citep[Theorem 3.3]{LS_DMV_strong_uniqueness}. By a \textit{strong solution} to \eqref{original_1}--\eqref{viscosity_coefficients} we mean a solution
$(\fatu,\varrho,\theta)$, where 
\begin{gather}
    \varrho,\theta\in C([0,T];H^3(\Omega)), \qquad \varrho,\theta>0 \quad \text{in $Q_{T}$}, \qquad \varrho_t,\theta_t\in C([0,T];H^2(\Omega)), \notag \\[2mm]
    \fatu\in L^2(0,T;H^4(\Omega)^d)\cap C([0,T];H^3(\Omega)^d), \qquad \fatu_t\in L^2(0,T;H^2(\Omega)^d), \notag \\[2mm]
    \fatu_{tt}\in L^2(0,T;L^2(\Omega)^d). \notag 
\end{gather}
As far as we are aware, there is no existence result for strong solutions to the compressible Navier-Stokes equations with potential temperature transport \eqref{original_1}--\eqref{original_4}, \eqref{pressure}--\eqref{viscosity_coefficients} with the no-slip boundary conditions \eqref{original_7} available in the literature. The only available existence result for this system can be found in the recent work of Zhai, Li and Zhou \citep{Zhai_Li_Zhou}, where the Cauchy problem is studied. Both a local and a global existence result are stated. The strong solutions belong to Besov spaces and are obtained by applying techniques from harmonic analysis. 

The aim of the present paper is to reveal the importance of the DMV-strong uniqueness principle. To this end, we first show that there do exist local-in-time strong solutions \eqref{original_1}--\eqref{viscosity_coefficients} following the approach in \citep{Valli_Periodic_and_stationary}, \citep{Valli_Zajaczkowski}. Afterwards, we prove the subsequent blow-up criterion for the strong solution $(\fatu,\varrho,\theta)$ to \eqref{original_1}--\eqref{viscosity_coefficients} under the assumption $\fatb=\bm{0}$\footnote{This assumption is only made to simplify the presentation; see also Section \ref{sec_conclusion}.}: 

If the maximal existence time $\Tmax$ of the strong solution $(\fatu,\varrho,\theta)$ to \eqref{original_1}--\eqref{viscosity_coefficients} is finite, then
\begin{align}
    \limsup_{t\,\uparrow\,\Tmax}\left(||\varrho(t)||_{L^\infty} + ||\fatu(t)||_{L^\infty}\right) = \infty. \notag
\end{align}
Consequently, the strong solution can always be extended beyond some given time $T'$ provided its density and velocity components are bounded on $(0,T')\times\Omega$. In the past years, a great variety of blow-up criteria for the solutions to different Navier-Stokes systems have been studied; see, e.g., \citep{Serrin}, \citep{Beirao_da_Veiga}, \citep{Bae_Kang} for the incompressible Navier-Stokes equations, \citep{CCK}, \citep{Huang_Xin}, \citep{SWZ} for the barotropic Navier-Stokes equations and \citep{SWZ_NSF}, \citep{Basaric_Feireisl_Mizerova}, \citep{Feireisl_Wen_Zhu} for the Navier-Stokes-Fourier equations. We further recall that already in 1958 John Nash conjectured that boundedness of certain quantities, such as density and temperature, is sufficient to guarantee the continuability in time of strong solutions to flow equations; see \citep{Nash} and Section \ref{sec_Nash}. Thus, our blow-up criterion can be considered as a proof of Nash's hypothesis for the compressible Navier-Stokes equations with potential temperature transport.

Since it is technically easier, we first prove the abovementioned results for the following system that arises from \eqref{original_1}--\eqref{viscosity_coefficients} by introducing the new variable $Z\equiv\varrho\theta$:
\begin{alignat}{2}
    \varrho_t + \fatu\cdot\gradx\varrho + \varrho\,\divx(\fatu) &= 0 && \qquad\text{in $Q_T$,} \label{system_1} \\[2mm]
    \varrho[\fatu_t + (\fatu\cdot\gradx)\fatu] -L\fatu + \gradx p(Z) &= \varrho\fatb && \qquad\text{in $Q_T$,} \label{system_2} \\[2mm]
    Z_t + \fatu\cdot\gradx Z + Z\divx(\fatu) &= 0 && \qquad\text{in $Q_T$,} \label{system_3} \\[2mm]
    (\fatu,\varrho,Z)(0,\cdot) &= (\fatu_0,\varrho_0,Z_0) && \qquad \text{in $\Omega$,} \label{system_4} \\[2mm]
    \fatu|_{[0,T]\times\po} &= \bm{0}. && \label{system_7}
\end{alignat}
Here, we have denoted by $L$ the \textit{Lamé operator}
\begin{align}
    L = \mu\deltax + (\lambda+\mu)\gradx\divx, \label{lame_operator}
\end{align}
where the viscosity coefficients satisfy \eqref{viscosity_coefficients}, and the pressure now reads
\begin{align}
    p(Z) = aZ^{\gamma}, \qquad a>0, \quad \gamma>1. \label{pressure_with_Z}
\end{align}
As we only consider strictly positive initial densities $\varrho_0$ and potential temperatures $\theta_0$, the results for the transformed system then carry over to the original system \eqref{original_1}--\eqref{viscosity_coefficients}. 

The paper is organized as follows: In Section \ref{sec_notation}, we present our notation and state the conventions used throughout the paper. Sections \ref{sec_local_in_time_strong_solutions} and \ref{sec_conditional_regularity} form the heart of the paper. This is where we prove the existence of local-in-time strong solutions and the blow-up criterion for the transformed system \eqref{system_1}--\eqref{pressure_with_Z}, \eqref{viscosity_coefficients}, respectively. In Section \ref{sec_conclusion}, we formulate our results for the original system \eqref{original_1}--\eqref{viscosity_coefficients} and combine them with the DMV-strong uniqueness principle to draw connections to our previous results for the Navier-Stokes equations with potential temperature transport. In particular, combining the results from Section \ref{sec_conclusion} with our previous results from \citep{LS_existence} and \citep{LS_DMV_strong_uniqueness}, we will obtain conditional convergence for any consistent numerical scheme for the Navier-Stokes equations with potential temperature transport. More precisely, if the initial data are sufficiently smooth and the numerical solutions are uniformly bounded with respect to a discretization parameter, then the strong solution to the Navier-Stokes equations with potential temperature exists and the numerical solutions converge strongly to it.

\section{General notation and conventions}\label{sec_notation}
We use the following notation and conventions.

\subsection{Dimension and domain}
The dimension in space is denoted by $d$. Unless otherwise stated, both $d=2$ and $d=3$ are admissible. Moreover, $\Omega$ always denotes a domain in $\reals^d$ and for $t>0$ we write $Q_t\equiv(0,t)\times\Omega$.

\subsection{Vectors and matrices}
All vectors $\fata\in\reals^d$ are interpreted as \textit{column vectors}, i.e. as $(d\times 1)$-matrices. The corresponding row vectors are denoted by $\fata^T$. For $\fata=(a_i)_i,\fatb=(b_i)_i\in\reals^d$ and $\matrixA=(a_{ij})_{ij},\matrixB=(b_{ij})_{ij}\in\reals^{d\times d}$, we write
\begin{gather}
    \fata\cdot\fatb \equiv \sum_{i\,=\,1}^d a_i b_i, \qquad |\fata|^2 \equiv \fata\cdot\fata, \qquad \matrixA:\matrixB \equiv \sum_{i,j\,=\,1}^d a_{ij}b_{ij}, \qquad |\matrixA|^2 \equiv \matrixA:\matrixA, \notag \\[2mm] \fata\otimes\fatb \equiv (a_i b_j)_{ij}, \qquad \matrixA\otimes\fatb \equiv (a_{ij} b_k)_{ijk} \qquad \text{and} \qquad \fatb\otimes\matrixA \equiv (b_i a_{jk})_{ijk}. \notag
\end{gather}
The transpose of the matrix $\matrixA$ is denoted by $\matrixA^T\equiv(a_{ji})_{ij}$. Matrix multiplication of $\matrixA$ and $\matrixB$ is written as $\matrixA\matrixB$ or, for better readability, as $\matrixA\cdot\matrixB$.

\subsection{Differential operators}
Let $(f,\fatf,\matrixF)=(f,(f_i)_i,(f_{ij})_{ij}):Q_T\to\reals\times\reals^d\times\reals^{d\times d}$, $(t,\fatx)\mapsto (f,\fatf,\matrixF)(t,\fatx)$. By $\pt f\equiv f_t$ and $\partial_{x_i}f\equiv\ppi f$ we denote the partial derivatives of $f$ with respect to $t$ and the $i$-th component of $\fatx$, respectively. $\deltax f$ denotes the Laplacian of $f$ with respect to the spatial variables and $\dot{f}$ the \textit{material derivative} of $f$, i.e.
\begin{align}
    \deltax f \equiv \sum_{i\,=\,1}^d \ppi^2 f, \qquad \dot{f}\equiv f_t + (\fatu\cdot\gradx)f \equiv f_t + \sum_{i\,=\,1}^d u_i\,\partial_{x_i} f, \notag
\end{align}
where $\fatu=(u_i)_i$ is the velocity component of the solution to \eqref{system_1}--\eqref{pressure_with_Z}, \eqref{viscosity_coefficients}. The operators $\pt,\partial_{x_i},\deltax$ as well as the material derivative are understood componentwise if applied to $\reals^d$-valued or $\reals^{d\times d}$-valued functions.
We further write
\begin{gather}
    \gradx f \equiv (\ppi f)_i, \qquad \gradx\fatf \equiv (\ppj f_i)_{ij}, \qquad \gradx^2 f \equiv (\ppi\ppj f)_{ij}, \qquad \gradx^2\fatf \equiv (\ppi\ppj f_k)_{ijk}, \notag \\[2mm]
    \divx(\fatf) \equiv \sum_{i\,=\,1}^d \ppi f_i, \qquad \divx(\matrixF)\equiv \left(\,\sum_{j\,=\,1}^d \ppj f_{ij}\right)_{\!i}. \notag
\end{gather}

\subsection{Integral means}
If $W\subset\reals^d$ is a Lebesgue measurable set and $f\in L^1(W)$, then we use the following notation for the integral mean of $f$ over $W$
\begin{align}
    f_W \equiv \fint_W f\;\dx \equiv \frac{1}{|W|}\int_W f\;dx, \notag
\end{align}
where $|W|$ denotes the Lebesgue measure of $W$.

\subsection{Norms}
Let $k\in\naturals$ and $p\in[1,\infty]$. For brevity we write $||\cdot||_{L^p}$ instead of $||\cdot||_{L^p(\Omega)}$ or $||\cdot||_{L^p(\Omega)^d}$ or $||\cdot||_{L^p(\Omega)^{d\times d}}$ or even $||\cdot||_{L^p(\Omega)^{d\times d\times d}}$. The shorthands $||\cdot||_{H^k}$ and $||\cdot||_{W^{k,p}}$ are used analogously.

Moreover, if $V$ is a normed vector space, we write $||\cdot||_{L^p V}$ instead of $||\cdot||_{L^p(0,t;V)}$ provided the time instant $t>0$ is clear from the context.

Finally, in some places we write $||p||_{C^k}$. By this we mean $||p||_{C^k([r,R])}$ -- the lower bound $r>0$ and the upper bound $R>r$ for the pressure argument $Z$ being clear from the context. 

\subsection{Constants}
Unless otherwise stated, all constants $c,c_i$, $i\in\{0,\dots,3\}$, appearing in Section \ref{sec_local_in_time_strong_solutions} depend at most on $d,\Omega,\mu,\lambda,\gamma,a,\varrho_0,Z_0,\fatu_0$ and the bounds $m,M$. In particular, they do not depend on the final times $T,T^\star$.

Unless otherwise stated, the constants $C_j$, $j\in\{0,\dots,6\}$, appearing mainly in Section \ref{sec_conditional_regularity} depend at most on $d,\Omega,\mu,\lambda$ and the integrability constant $q$. Constants named $C$ may additionally depend on $\gamma,a,\varrho_0,Z_0,\fatu_0$, the bounds $m,M$, the maximal existence time $\Tmax$ of the strong solution and on the uniform bound $K$ in \eqref{assumption}. In particular, the constants $C,C_j$, $j\in\{0,\dots,6\}$, do not depend on the specific time instants $t,T\in[0,\Tmax)$.

In the sequel, the above dependencies will be suppressed in the notation. To keep the notation simple, we further allow the meaning of the unnumbered constants $c,C$ (especially used in proofs) to vary from line to line.

\section{Local-in-time strong solutions}\label{sec_local_in_time_strong_solutions}
In this section, we prove existence and uniqueness of local-in-time strong solutions to the compressible Navier-Stokes equations with potential temperature transport \eqref{system_1}--\eqref{pressure_with_Z}, \eqref{viscosity_coefficients}. To this end, we follow the strategy proposed in \citep{Valli_Periodic_and_stationary} and \citep{Valli_Zajaczkowski}. For the existence part, we choose a suitable linearization to split the problem into three separate subproblems. In Section \ref{sec_lin_mom_eq}, we prove the unique solvability of the linearized momentum equation and obtain a priori bounds for it. Section \ref{sec_lin_dens_eq} is devoted to the analysis of the linearized continuity equation and the linearized $Z$-equation and in Section \ref{sec_existence} we apply Schauder's fixed-point theorem to prove the existence of local-in-time strong solutions to \eqref{system_1}--\eqref{pressure_with_Z}, \eqref{viscosity_coefficients}. Finally, we show the uniqueness of strong solutions in Section \ref{sec_uniqueness}.

\subsection{Linearized momentum equation}\label{sec_lin_mom_eq}
In this section, we analyze the following linearized version of the momentum equation:

\begin{alignat}{2}
    \trho\fatu_t - L\fatu &= \fatF && \qquad \text{in $Q_T$,} \label{linear_momentum1} \\[2mm]
    \fatu|_{[0,T]\times\po} &= \bm{0}, && \label{linear_momentum2} \\[2mm]
    \fatu(0,\cdot) &= \fatu_0 && \qquad \text{in $\Omega$.} \label{linear_momentum3}
\end{alignat}
Here, $\trho$ and $\fatF$ are given functions.

\subsubsection{Elliptic regularity for the Lamé system}
To obtain suitable a priori bounds for the solution to \eqref{linear_momentum1}--\eqref{linear_momentum3}, we rely on the subsequent lemma concerning solutions to the Lamé system
\begin{align}
    L\fatU = \mu\deltax\fatU + (\lambda+\mu)\gradx\divx(\fatU) &= \fatF \qquad \text{in $\Omega$}, \label{lame0} \\
    \fatU|_{\po} &= \bm{0}. \label{lame01}
\end{align}
Before formulating the result, let us recall that for any $\fatF\in L^2(\Omega)^d$ the Lamé system possesses a unique weak solution $\fatU\in H^1_0(\Omega)^d$. This is an immediate consequence of the Lax-Milgram lemma (cf. \citep[§6.2.1, Theorem 1]{Evans}). Moreover, seeing that 
\begin{align}
    (L\fatU)_{\alpha} = \partial_{\ell}(a^{\alpha\beta}_{\ell k}\ppk U_{\!\beta}), \quad \text{where} \quad a^{\alpha\beta}_{\ell k} = \mu\delta_{\ell k}\delta_{\alpha\beta} + (\lambda+\mu)\delta_{\ell\alpha}\delta_{k\beta}, \quad \delta_{\alpha\beta} = \left\{
    \begin{array}{rl}
        1 & \text{if $\alpha=\beta$,} \\
        0 & \text{else,}
    \end{array}\right. \notag
\end{align}
it follows from \citep[Proposition 3.3 with $r=0$]{Mitrea_Monniaux} that the Lamé operator is \textit{strictly elliptic}, i.e. there exists a constant $\kappa>0$ such that
\begin{align}
    \sum_{\alpha,\beta,\ell,k\,=\,1}^d a^{\alpha\beta}_{\ell k}\zeta_{\ell}^{\alpha}\zeta_{k}^{\beta} \geq \kappa|\zeta|^2  \notag
\end{align}
for all $\zeta=(\zeta^{\alpha}_{\ell})_{\ell\alpha}\in\reals^{d\times d}$. Therefore, the proof of the subsequent lemma can be performed following the steps in \citep[§6.3]{Evans}.

\begin{lemma}\label{elliptic_regularity}
    Let $k\in\naturals_0$ and $\Omega\subset\reals^d$ a bounded domain with boundary $\po\in C^{k+2}$. The Lamé operator $L$ is an isomorphism $H^{k+2}(\Omega)^d\cap H^1_0(\Omega)^d\to H^k(\Omega)^d$. In particular, there exists a constant $C_0>0$\footnote{The constant $C_0$ in Lemma \ref{elliptic_regularity} equals the constant $C_0$ in Lemma \ref{lame_regularity} for $q=2$.} such that
    \begin{align}
        ||L^{-1}\fatF||_{H^{k+2}} \leq C_0||\fatF||_{H^k} \notag
    \end{align}
    for all $\fatF\in H^k(\Omega)^d$.
\end{lemma}
In the sequel, we will often write ``by elliptic regularity'' instead of explicitly referring to Lemma \ref{elliptic_regularity}.

\subsubsection{Solutions to the linearized momentum equation}
We begin our analysis of the linearized momentum equation by recording the following lemma; cf. \citep[Lemma 2.1]{Valli_Periodic_and_stationary}.

\begin{lemma}
    Let $\po\in C^2$, $T>0$, $0<m<M$, $\trho\in L^\infty(Q_T)$, $m/2\leq\trho\leq 2M$ a.e. in $Q_T$, $\fatF\in L^2(0,T;L^2(\Omega)^d)$ and $\fatu_0\in H_0^1(\Omega)^d$. Then there exists a unique solution $\fatu\in L^2(0,T;H^2(\Omega)^d)\cap C([0,T];H_0^1(\Omega)^d)$ to \eqref{linear_momentum1}--\eqref{linear_momentum3} such that $\fatu_t\in L^2(0,T;L^2(\Omega)^d)$ and 
    \begin{align}
        \mu ||\gradx\fatu||_{L^\infty L^2}^2 + \frac{m}{32M^2}||L\fatu||_{L^2L^2}^2 + \frac{m}{2}||\fatu_t||_{L^2L^2}^2 \leq (8\mu+6\lambda)||\gradx\fatu_0||_{L^2}^2 + 2\left(\frac{4}{m} + \frac{m}{16M^2}\right)||\fatF||_{L^2L^2}^2. \notag
    \end{align}
\end{lemma}

In order to show higher regularity of the solution to \eqref{linear_momentum1}--\eqref{linear_momentum3}, we consider the equation 
\begin{align}
    \trho\hspace{0.2mm}\fatV_t - L\fatV + \trho_t\fatV = \fatG \qquad \text{in $Q_T$}\label{derivative_linear_momentum1} 
\end{align}
that can be formally obtained by taking the derivative of \eqref{linear_momentum1} with respect to $t$ and by setting $\fatV = \fatu_t$ and $\fatG=\fatF_t$. The basic idea is to find a solution to \eqref{derivative_linear_momentum1} and to integrate it in time to obtain a solution to \eqref{linear_momentum1}. For the moment, however, we shall treat equation \eqref{derivative_linear_momentum1} as independent and we equip it with the initial and boundary conditions
\begin{alignat}{2}
    \fatV|_{[0,T]\times\po} &= \bm{0}, && \label{derivative_linear_momentum2} \\[2mm]
    \fatV(0,\cdot) &= \fatV_0 && \qquad \text{in $\Omega$.} \label{derivative_linear_momentum3}
\end{alignat}

\begin{lemma}\label{lemma_improved_regularity}
    Let $\po\in C^2$, $T>0$, $0<m<M$, $\trho\in L^\infty(Q_T)$, $m/2\leq\trho\leq 2M$ a.e. in $Q_T$, $\trho(0,\cdot)\in L^\infty(\Omega)$, $m\leq\trho (0,\cdot)\leq M$ a.e. in $\Omega$, $\trho_t\in L^2(0,T;L^3(\Omega))$,  $\fatG\in L^2(0,T;L^2(\Omega)^d)$, and $\fatV_0\in H_0^1(\Omega)^d$. Then there exists a unique solution $\fatV\in L^2(0,T;H^2(\Omega)^d)\cap C([0,T];H^1_0(\Omega)^d)$ to \eqref{derivative_linear_momentum1}--\eqref{derivative_linear_momentum3} with $\fatV_t\in L^2(0,T;L^2(\Omega)^d)$. Moreover, there exists a constant $c_0>0$ such that  
    \begin{align}
        ||\fatV||_{L^2 H^2}^2 + ||\fatV||_{L^\infty H^1}^2 + ||\fatV_t||_{L^2 L^2}^2 \leq c_0\left(||\fatG||_{L^2L^2}^2 + ||\fatV_0||_{H^1}^2\right)(1+||\trho_t||_{L^2L^3}^2)\exp(c_0||\trho_t||_{L^2L^3}^2) \label{estimate_U}
    \end{align}
    for all $(T,\trho,\fatG,\fatV_0,\fatV)$ satisfying \eqref{derivative_linear_momentum1}--\eqref{derivative_linear_momentum3} and having the aforementioned properties.
\end{lemma}

\begin{proof}
    Since the structure of \eqref{derivative_linear_momentum1} is almost the same as that of \eqref{linear_momentum1}, we use the same strategy for the proof as was used for the proof of \citep[Lemma 2.1]{Valli_Periodic_and_stationary}. We start by proving a priori estimate \eqref{estimate_U}. Testing \eqref{derivative_linear_momentum1} with $\fatV_t - \varepsilon_0 L\fatV$ (with $\varepsilon_0$ to be chosen suitably), integrating over $\Omega$ and performing integration by parts on the term $\into \fatV_t\cdot L\fatV\;\dx$, we obtain (cf. the proof of \citep[Lemma 2.1]{Valli_Periodic_and_stationary})\footnote{Note that this computation is formal since we have no information concerning the differentiability of $\gradx\fatV$ with respect to time. The rigorous proof can be performed using a time regularization of $\fatV$.}
    \begin{align}
        &\frac{m}{2}\,||\fatV_t||_{L^2}^2 + \mu\,\ddt\,||\gradx\fatV||_{L^2}^2 + (\lambda+\mu)\,\ddt\,||\divx(\fatV)||_{L^2}^2 + \frac{m}{32M^2}\,||L\fatV||_{L^2}^2 \notag \\[2mm]
        &\quad\leq \left(\frac{4}{m}+\frac{m}{16M^2}\right)||\fatG||_{L^2}^2 + \left|\into\trho_t\fatV\cdot(\fatV_t - \varepsilon_0 L\fatV)\;\dx\right| \notag \\[2mm]
        &\quad\leq c||\fatG||_{L^2}^2 + ||\trho_t||_{L^3}||\fatV||_{L^6}\big(||\fatV_t||_{L^2} + \varepsilon_0||L\fatV||_{L^2}\big) \notag \\[2mm]
        &\quad\leq c\left(||\fatG||_{L^2}^2 + ||\trho_t||_{L^3}^2||\fatV||_{L^6}^2\right) + \frac{m}{4}\,||\fatV_t||_{L^2}^2 + \frac{m}{64M^2}\,||L\fatV||_{L^2}^2\,. \notag 
    \end{align}
    In view of Poincaré's inequality, the above inequality implies
    \begin{align}
        &||\fatV_t||_{L^2}^2 + \tmu\,\ddt\,||\gradx\fatV||_{L^2}^2 + (\tlambda+\tmu)\,\ddt\,||\divx(\fatV)||_{L^2}^2 + ||L\fatV||_{L^2}^2 \notag \\[2mm]
        &\qquad
        \leq c\left(||\fatG||_{L^2}^2 + ||\trho_t||_{L^3}^2(\tmu||\gradx\fatV||_{L^2}^2 + (\tlambda+\tmu)||\divx(\fatV)||_{L^2}^2)\right) \label{estimate1}  
    \end{align}
    with certain constants $\tmu = \tmu(\mu,m,M)>0$, $\tlambda=\tlambda(\lambda,m,M)>0$.
    Applying Gronwall's lemma and Poincaré's inequality, we deduce
    \begin{align}
        ||\fatV||_{L^\infty H^1}^2 \leq c\left(||\fatG||_{L^2L^2}^2 + ||\fatV_0||_{H^1}^2\right)\exp(c||\trho_t||_{L^2L^3}^2). \label{estimate2}
    \end{align}
    Furthermore, integrating \eqref{estimate1} over $(0,T)$ and using the standard elliptic theory as well as \eqref{estimate2}, we get
    \begin{align}
        ||\fatV_t||_{L^2L^2}^2 + ||\fatV||_{L^2H^2}^2 &\leq c\left(||\fatV_t||_{L^2L^2}^2 + ||L\fatV||_{L^2L^2}^2\right) \notag \\[2mm]
        &\leq c\left(||\gradx\fatV_0||_{L^2}^2 + ||\fatG||_{L^2L^2}^2 + \int_0^{\,T} ||\trho_t||_{L^3}^2||\gradx\fatV||_{L^2}^2\;\dt\right) \notag \\[2mm]
        &\leq c\left(||\fatG||_{L^2L^2}^2 + ||\fatV_0||_{H^1}^2\right)(1+||\trho_t||_{L^2L^3}^2)\exp(c||\trho_t||_{L^2L^3}^2)\,. \notag
    \end{align}
    Next, we prove the existence of a solution to \eqref{derivative_linear_momentum1}--\eqref{derivative_linear_momentum3}. In the case $\trho\equiv\brho$ it is well known that for every pair $(\fatG,\fatV_0)\in L^2(0,T;L^2(\Omega)^d)\times H^1_0(\Omega)^d$ there exists a solution in the desired regularity class; see, e.g., \citep[Chapter VII, \S10, Theorem 10.4]{Ladyzhenskaya}. To prove the existence of a solution for an arbitrary $\trho$, we employ the method of continuity (cf. Theorem \ref{method_of_continuity} in the appendix). We define the Banach spaces $(\mathcal{H},||\cdot||_{\mathcal{H}})$, $(\mathcal{Y},||\cdot||_{\mathcal{Y}})$ via
    \begin{alignat}{3}
        \mathcal{H} &= \{\fatV\in L^2(0,T;H^2(\Omega)^d\cap H^1_0(\Omega)^d)\,|\,\fatV_t\in L^2(0,T;L^2(\Omega)^d)\}, & \qquad ||\fatV||_{\mathcal{H}}^2 &= \text{LHS of \eqref{estimate_U}}, \notag \\[2mm]
        \mathcal{Y} &= L^2(0,T;L^2(\Omega)^d)\times H^1_0(\Omega)^d, & \qquad ||(\fatG,\fatV_0)||_{\mathcal{Y}}^2 &= \text{RHS of \eqref{estimate_U}}. \notag
    \end{alignat}
    Furthermore, for $\alpha\in[0,1]$ we set 
    \begin{align}
        \varrho_\alpha = (1-\alpha)\brho + \alpha\trho \qquad \text{and} \qquad T_\alpha = (1-\alpha)T_0 + \alpha T_1, \notag
    \end{align}
    where $T_0,T_1:\mathcal{H}\to\mathcal{Y}$,
    \begin{align}
        T_0(\fatV) = (\brho\hspace{0.4mm}\fatV_t - L\fatV, \fatV(0)), \qquad T_1(\fatV) = (\trho\hspace{0.4mm}\fatV_t - L\fatV + \trho_t\fatV, \fatV(0)), \notag
    \end{align}
    i.e. $T_\alpha(\fatV)=(T_{\alpha,1}(\fatV),T_{\alpha,2}(\fatV))=(\varrho_\alpha\fatV_t - L\fatV+(\varrho_\alpha)_t\fatV,\fatV(0))$. 
    Clearly, an operator $T_\alpha$ is surjective if and only if \eqref{derivative_linear_momentum1}--\eqref{derivative_linear_momentum3} with $\trho$ replaced by $\varrho_\alpha$ is solvable for all $(\fatG,\fatV_0)\in L^2(0,T;L^2(\Omega)^d)\times H^1_0(\Omega)^d$. Hence, $T_0$ is surjective. In order to show that $T_1$ is surjective as well, we need to show that $(T_\alpha)_{\alpha\,\in\,[0,1]}$ is a norm continuous family of bounded linear operators $\mathcal{H}\to\mathcal{Y}$ and that there exists $c>0$ such that
    \begin{align}
        ||\fatV||_{\mathcal{H}}\leq c||T_\alpha(\fatV)||_{\mathcal{Y}} \notag
    \end{align}
    for all $\alpha\in[0,1]$ and all $\fatV\in\mathcal{H}$. Obviously, the operators $(T_\alpha)_{\alpha\,\in\,[0,1]}$ are linear. Their boundedness follows from the observation that
    \begin{align}
        &||T_\alpha(\fatV)||_{\mathcal{Y}}^2 = c_0\left(||\varrho_\alpha\fatV_t - L\fatV + (\varrho_\alpha)_t\fatV||_{L^2L^2}^2 + ||\fatV(0)||_{H^1}^2\right)(1+||\trho_t||_{L^2L^3}^2)\exp(c_0||\trho_t||_{L^2L^3}^2) \notag \\[2mm]
        &\leq c\left(||\fatV_t||_{L^2L^2}^2 + ||\fatV||_{L^2H^2}^2 + ||(\varrho_\alpha)_t||_{L^2L^3}^2||\fatV||_{L^\infty L^6}^2 + ||\fatV||_{L^\infty H^1}^2\right)(1+||\trho_t||_{L^2L^3}^2)\exp(c_0||\trho_t||_{L^2L^3}^2) \notag \\[2mm]
        &\leq c\left(||\fatV_t||_{L^2L^2}^2 + ||\fatV||_{L^2H^2}^2 + ||\trho_t||_{L^2L^3}^2||\fatV||_{L^\infty H^1}^2 + ||\fatV||_{L^\infty H^1}^2\right)(1+||\trho_t||_{L^2L^3}^2)\exp(c_0||\trho_t||_{L^2L^3}^2) \notag \\[2mm]
        &\leq c(1+||\trho_t||_{L^2L^3}^2)^2\exp(c_0||\trho_t||_{L^2L^3}^2)||\fatV||_{\mathcal{H}}^2 \notag
    \end{align}
    for all $\alpha\in[0,1]$ and all $\fatV\in\mathcal{H}$.
    The continuity of the map $[0,1]\ni\alpha\mapsto T_\alpha\in\mathcal{L}(\mathcal{X},\mathcal{Y})$ is a consequence of 
    \begin{align}
        ||(T_{\alpha_1}-T_{\alpha_2})(\fatV)||_{\mathcal{Y}}^2 &= c_0\big|\big|(\alpha_1-\alpha_2)[(\trho-\brho)\fatV_t + \trho_t\fatV]\big|\big|_{L^2L^2}^2(1+||\trho_t||_{L^2L^3}^2)\exp(c_0||\trho_t||_{L^2L^3}^2) \notag \\[2mm]
        &\leq c|\alpha_1-\alpha_2|^2(||\fatV_t||_{L^2L^2}^2 + ||\trho_t\fatV||_{L^2L^2}^2)(1+||\trho_t||_{L^2L^3}^2)\exp(c_0||\trho_t||_{L^2L^3}^2) \notag \\[2mm]
        &\leq c|\alpha_1-\alpha_2|^2(||\fatV_t||_{L^2L^2}^2 + ||\trho_t||_{L^2L^3}^2||\fatV||_{L^\infty L^6}^2)(1+||\trho_t||_{L^2L^3}^2)\exp(c_0||\trho_t||_{L^2L^3}^2) \notag \\[2mm]
        &\leq c|\alpha_1-\alpha_2|^2(1+||\trho_t||_{L^2L^3}^2)^2\exp(c_0||\trho_t||_{L^2L^3}^2)||\fatV||_{\mathcal{H}}^2 \notag
    \end{align}
    for all $\alpha_1,\alpha_2\in[0,1]$ and all $\fatV\in\mathcal{H}$.
    Next, let us fix an arbitrary pair $(\alpha,\fatV)\in[0,1]\times\mathcal{H}$. Since $\fatV$ satisfies \eqref{derivative_linear_momentum1}--\eqref{derivative_linear_momentum3} with $\trho$ replaced by $\varrho_\alpha$ and $(\fatG,\fatV_0)=T_\alpha(\fatV)$, the a priori estimate yields
    \begin{align}
        ||\fatV||_{\mathcal{H}}^2 &\leq c_0\left(||T_{\alpha,1}(\fatV)||_{L^2L^2}^2 + ||\fatV(0)||_{H^1}^2\right)(1+||(\varrho_\alpha)_t||_{L^2L^3}^2)\exp(c_0||(\varrho_\alpha)_t||_{L^2L^3}^2) \notag \\[2mm]
        &\leq c_0\left(||T_{\alpha,1}(\fatV)||_{L^2L^2}^2 + ||\fatV(0)||_{H^1}^2\right)(1+||\trho_t||_{L^2L^3}^2)\exp(c_0||\trho_t||_{L^2L^3}^2) = ||T_\alpha(\fatV)||_{\mathcal{Y}}^2. \notag
    \end{align}
    Hence, the method of continuity is applicable and implies that $T_1$ is surjective or, equivalently, that for every $(\fatG,\fatV_0)\in\mathcal{H}$ there exists a solution $\fatV\in\mathcal{Y}$ to \eqref{derivative_linear_momentum1}--\eqref{derivative_linear_momentum3}. Finally, the proof of uniqueness is trivial.  
\end{proof}

Then, following the idea of integrating in time the solution $\fatV$ to \eqref{derivative_linear_momentum1}--\eqref{derivative_linear_momentum3}, we obtain the following corollary.

\begin{corollary}\label{corollary_solutions_linearized_momentum_equation}
    Let $\po\in C^4$, $T>0$, $0<m<M$, $\trho\in L^\infty(Q_T)$, $m/2\leq\trho\leq 2M$ a.e. in $Q_T$, $\trho(0,\cdot)\in L^\infty(\Omega)$, $m\leq\trho (0,\cdot)\leq M$ a.e. in $\Omega$, $\trho\in L^\infty(0,T;H^2(\Omega)^d)$, $\trho_t\in L^2(0,T;L^3(\Omega))$,  $\fatF\in L^2(0,T;H^2(\Omega)^d)$, $\fatF_t\in L^2(0,T;L^2(\Omega)^d)$, $[\fatF(0)+ L\fatu_0]/\trho(0)\in H_0^1(\Omega)^d$ and $\fatu_0\in H^3(\Omega)^d\cap H_0^1(\Omega)^d$. Then the solution to \eqref{linear_momentum1}--\eqref{linear_momentum3} is such that $\fatu\in L^2(0,T;H^4(\Omega)^d)\cap C([0,T];H^3(\Omega)^d)$, $\fatu_t\in L^2(0,T;H^2(\Omega)^d)\cap C([0,T];H^1(\Omega)^d)$, $\fatu_{tt}\in L^2(0,T;L^2(\Omega)^d)$.  Moreover, there exists a constant $c_1>0$ such that
    \begin{align}
        &||\fatu||_{L^\infty H^3}^2 + ||\fatu||_{L^2H^4}^2 + ||\fatu_t||_{L^\infty H^1}^2 + ||\fatu_t||_{L^2 H^2}^2 + ||\fatu_{tt}||_{L^2 L^2}^2 \notag \\[2mm]
        &\leq c_1\;\Bigg\{||\fatF||_{L^2H^2}^2 + ||\fatF||_{L^\infty H^1}^2 + \left(||\fatF_t||_{L^2L^2}^2 + \left|\left|\frac{\fatF(0) + L\fatu_0}{\trho(0)}\right|\right|_{H^1}^2\right) \notag \\[2mm]
        &\qquad\cdot(1 + ||\trho_t||_{L^2L^3}^2)(1 + ||\gradx\trho||_{L^\infty H^1}^2)\exp(c_1||\trho_t||_{L^2L^3}^2)\Bigg\} \label{u_problem_2}
    \end{align}
    for all $(T,\trho,\fatF,\fatu_0,\fatu)$ satisfying \eqref{linear_momentum1}--\eqref{linear_momentum3} and having the aforementioned properties.
\end{corollary}

\begin{proof}
    Let $\fatV$ be the solution to \eqref{derivative_linear_momentum1}--\eqref{derivative_linear_momentum3} with $\fatG=\fatF_t$ and $\fatV_0 = (\fatF(0) + L\fatu_0)/\trho(0)$ from Lemma \ref{lemma_improved_regularity}. Then,
    \begin{align}
        &\fatu(t) \equiv \fatu_0 + \int_0^{\,t} \fatV(s)\;\ds \in H^1(0,T;H^2(\Omega)^d) \hookrightarrow C([0,T];H^2(\Omega)^d) \notag
    \end{align}
    satisfies $\fatu(0) = \fatu_0$, $\fatu|_{[0,T]\times\po} = \bm{0}$, $\fatu_t = \fatV \in L^2(0,T;H^2(\Omega)^d)\cap C([0,T];H^1_0(\Omega)^d)$, $\fatu_{tt} = \fatV_t\in L^2(0,T;L^2(\Omega)^d)$
    and 
    \begin{align}
        (\trho\fatu_t- L\fatu)(t) &= (\trho\hspace{0.2mm}\fatV)(t) - L\left(\fatu_0 + \int_0^{\,t} \fatV(s)\;\ds\right) = (\trho\hspace{0.2mm}\fatV)(t) - L\fatu_0 - \int_0^{\,t} L\fatV(s)\;\ds \notag \\[2mm]
        &=(\trho\hspace{0.2mm}\fatV)(t) + \int_{0}^{\,t} (\fatF_t - \trho_t\fatV - \trho\hspace{0.2mm}\fatV_t)(s)\;\ds = (\trho\hspace{0.2mm}\fatV)(t) - L\fatu_0 + \int_{0}^{\,t} \frac{\mathrm{d}(\fatF - \trho\hspace{0.2mm}\fatV)}{\mathrm{d}t}\,(s)\;\ds \notag \\[2mm]
        &= (\trho\hspace{0.2mm}\fatV)(t) - L\fatu_0 + \fatF(t) - (\trho\hspace{0.2mm}\fatV)(t) - \fatF(0) + (\trho\fatV)(0) = \fatF(t) \notag
    \end{align}
    for all $t\in[0,T]$. That is, $\fatu$ solves \eqref{linear_momentum1}--\eqref{linear_momentum3}. To prove \eqref{u_problem_2}, we observe that \eqref{estimate_U} yields
    \begin{align}
        ||\fatu_t||_{L^2 H^2}^2 + ||\fatu_t||_{L^\infty H^1}^2 + & {}||\fatu_{tt}||_{L^2 L^2}^2 = ||\fatV||_{L^2 H^2}^2 + ||\fatV||_{L^\infty H^1}^2 + ||\fatV_t||_{L^2 L^2}^2 \notag \\[2mm]
        &\leq c\left(||\fatF_t||_{L^2L^2}^2 + \left|\left|\frac{\fatF(0) + L\fatu_0}{\trho(0)}\right|\right|_{H^1}^2\right)(1+||\trho_t||_{L^2L^3}^2)\exp(c||\trho_t||_{L^2L^3}^2). \label{estimate_ut_L2H2}
    \end{align}
    Moreover, by the assumptions on $\fatF,\trho$, we have $\fatF-\trho\hspace{0.2mm}\fatu_t\in L^2(0,T;H^2(\Omega)^d)$ and thus $\fatu\in L^2(0,T;H^4(\Omega)^d)$ by elliptic regularity. 
    Additionally,    
    \begin{align}
        ||\fatu||_{L^2H^4}^2 &\leq c||L\fatu||_{L^2H^2}^2 = c||\fatF-\trho\hspace{0.2mm}\fatV||_{L^2H^2}^2 \leq c\left(||\fatF||_{L^2H^2}^2 + ||\trho\hspace{0.2mm}\fatV||_{L^2H^2}^2\right) \notag \\[2mm]
        &\leq c\left(||\fatF||_{L^2H^2}^2 + ||\trho\hspace{0.2mm}\fatV||_{L^2L^2}^2 + ||\trho\hspace{0.2mm}\gradx\fatV||_{L^2L^2}^2 + ||\gradx\trho\otimes\fatV||_{L^2L^2}^2\right. \notag \\[2mm]
        &\qquad\left. {}+ ||\gradx^2\trho\otimes\fatV||_{L^2L^2}^2 + 2||\gradx\trho\otimes\gradx\fatV||_{L^2L^2}^2 + ||\trho\gradx^2\fatV||_{L^2L^2}^2 \right) \notag \\[2mm]
        &\leq c\left(||\fatF||_{L^2H^2}^2 + ||\fatV||_{L^2H^2}^2 + ||\gradx\trho||_{L^\infty L^3}^2||\fatV||_{L^2 L^6}^2 + ||\gradx^2\trho||_{L^\infty L^2}^2||\fatV||_{L^2 L^\infty}^2 
        \right. \notag \\[2mm]
        &\qquad\left. 
        {}+||\gradx\trho||_{L^\infty L^3}^2||\gradx\fatV||_{L^2 L^6}^2  \right) \notag \\[2mm]
        &\leq c\left(||\fatF||_{L^2H^2}^2 + (1 + ||\gradx\trho||_{L^\infty H^1}^2)||\fatV||_{L^2H^2}^2 \right), \notag \\[4mm]
        ||\fatu||_{L^\infty H^3}^2 &\leq c||L\fatu||_{L^\infty H^1}^2 = c||\fatF-\trho\hspace{0.2mm}\fatV||_{L^\infty H^1}^2 \leq c\left(||\fatF||_{L^\infty H^1}^2 + ||\trho\hspace{0.2mm}\fatV||_{L^\infty H^1}^2\right) \notag \\[2mm]
        &\leq c\left(||\fatF||_{L^\infty H^1}^2 + ||\trho\fatV||_{L^\infty L^2}^2 + ||\trho\gradx\fatV||_{L^\infty L^2}^2 + ||\gradx\trho\otimes\fatV||_{L^\infty L^2}^2\right) \notag \\[2mm]
        &\leq c\left(||\fatF||_{L^\infty H^1}^2 + ||\fatV||_{L^\infty H^1}^2 + ||\gradx\trho||_{L^\infty L^3}^2||\fatV||_{L^\infty L^6}^2\right) \notag \\[2mm]
        &\leq c\left(||\fatF||_{L^\infty H^1}^2 + (1 + ||\gradx\trho||_{L^\infty H^1}^2)||\fatV||_{L^\infty H^1}^2\right) \notag
    \end{align}
    and thus, by \eqref{estimate_U},
    \begin{align}
        ||\fatu||_{L^2H^4}^2 + ||\fatu||_{L^\infty H^3}^2 
        \leq c\;\Bigg\{&||\fatF||_{L^2H^2}^2 + ||\fatF||_{L^\infty H^1}^2 + \left(||\fatF_t||_{L^2L^2}^2 + \left|\left|\frac{\fatF(0) + L\fatu_0}{\trho(0)}\right|\right|_{H^1}^2\right) \notag \\[2mm]
        &\cdot(1 + ||\trho_t||_{L^2L^3}^2)(1 + ||\gradx\trho||_{L^\infty H^1}^2)\exp(c||\trho_t||_{L^2L^3}^2)\Bigg\}. \label{estimate_u_L2H4}
    \end{align}
    Together with \eqref{estimate_ut_L2H2} estimate \eqref{estimate_u_L2H4} yields \eqref{u_problem_2}.
\end{proof}

\subsection{Linearized continuity equation and linearized \texorpdfstring{$Z$}{Z}-equation}\label{sec_lin_dens_eq}
The linearized continuity equation reads
\begin{alignat}{2}
    \varrho_t + \tfatu\cdot\gradx\varrho + \varrho\,\divx(\tfatu) &= 0 && \qquad \text{in $Q_T$,} \label{linear_density1} \\[2mm]
    \varrho(0,\cdot) &= \varrho_0 && \qquad \text{in $\Omega$.} \label{linear_density2}
\end{alignat}
Analogously, the linearized $Z$-equation reads
\begin{alignat}{2}
    Z_t + \tfatu\cdot\gradx Z + Z\,\divx(\tfatu) &= 0 && \qquad \text{in $Q_T$,} \label{linear_pt1}\\[2mm]
    Z(0,\cdot) &= Z_0 && \qquad \text{in $\Omega$.} \label{linear_pt2}
\end{alignat}

Following the ideas of the proofs of \citep[Lemma 2.3]{Valli_Periodic_and_stationary} and \citep[Lemma 2.4]{Valli_Zajaczkowski}, we prove the following result for the linearized continuity equation \eqref{linear_density1}--\eqref{linear_density2}.

\begin{lemma}\label{lem_estimates_rho}
    Let $\po\in C^1$, $T>0$, $\tfatu\in L^1(0,T;H^4(\Omega)^d)\cap C([0,T];H^3(\Omega)^d)$ with $\tfatu|_{[0,T]\times\po}=\bm{0}$ and $\varrho_0\in H^3(\Omega)$. Then there exists a unique solution $\varrho\in C([0,T];H^3(\Omega))$ to \eqref{linear_density1}, \eqref{linear_density2} such that $\varrho_t\in C([0,T];H^2(\Omega))$. Moreover, there exist constants $c_2,c_3>0$ such that
    \begin{align}
        ||\varrho||_{L^\infty H^3} &\leq ||\varrho_0||_{H^3}\exp(c_2||\tfatu||_{L^1H^4})\,, \label{est_sigma} \\[2mm]
        ||\varrho_t||_{L^\infty H^2} &\leq c_3||\tfatu||_{L^\infty H^3}||\varrho_0||_{H^3}\exp(c_2||\tfatu||_{L^1H^4}) \label{est_sigmat}
    \end{align}
    for all $(T,\tfatu,\varrho_0,\varrho)$ satisfying \eqref{linear_density1}, \eqref{linear_density2} and having the aforementioned properties.
\end{lemma}

\begin{proof}
    The existence of the solution can be proven by means of the method of characteristics. 
    According to \citep[Lemma A.6]{Brezis}, the solution $\fatX(t,\fatx)$ to
    \begin{alignat}{2}
        \ddt\,\fatX(t,\fatx) &= \tfatu(t,\fatX(t,\fatx)), && \qquad \forall\,(t,\fatx)\in[0,T]\times\overline{\Omega}, \notag \\[2mm]
        \fatX(0,\fatx) &= \fatx, && \qquad \forall\,\fatx\in\overline{\Omega}, \notag
    \end{alignat}
    belongs to $C^1([0,T];\mathscr{D}^{3,2}_d(\Omega))$, where
    \begin{align*}
        \mathscr{D}^{3,2}_d(\Omega) = \Big\{\fateta\in H^3(\Omega)^d\,\Big|\,\fateta:\overline{\Omega}\to\overline{\Omega} \;\text{is bijective and}\;\fateta^{-1}\in H^3(\Omega)^d\Big\}.
    \end{align*}
    Applying the inverse function theorem to the map $Q_T\to Q_T$, $(t,\fatx)\mapsto(t,\fatX(t,\fatx))$, we see that $\fatX^{-1}\in C([0,T];H^3(\Omega)^d)$ and $\fatX^{-1}_t\in C([0,T];H^2(\Omega)^d)$. Consequently, the solution $\varrho$ to \eqref{linear_density1}, \eqref{linear_density2}, which is given by
    \begin{align}
        \varrho(t,\fatx) = \varrho_0(\fatX^{-1}(t,\fatx))\exp\left(-\int_0^{\,t} \divx(\tfatu)(s,\fatX(s,\fatX^{-1}(t,\fatx)))\;\ds\right), \label{solution_formula}
    \end{align}
    is an element of $C([0,T];H^3(\Omega)^d)$ and its time derivative $\varrho_t$ belongs to $ C([0,T];H^2(\Omega)^d)$.
    Next, we prove the a priori estimates. Testing \eqref{linear_density1} with $\varrho$ and integrating over $\Omega$ yields
    \begin{align}
        \frac{1}{2}\,\ddt\into\varrho^2\;\dx + \frac{1}{2}\into\tfatu\cdot\gradx(\varrho^2)\;\dx + \into\varrho^2\divx(\tfatu)\;\dx = 0\,. \notag
    \end{align}
    Integrating by parts the second term and using $\tfatu|_{[0,T]\times\po}=\bm{0}$, we get
    \begin{align}
        \frac{1}{2}\,\ddt\into\varrho^2\;\dx + \frac{1}{2}\into\divx(\tfatu)\,\varrho^2\;\dx = 0\,. \notag
    \end{align}
    Transferring the second term to the right-hand side and estimating it, we obtain
    \begin{align}
        \frac{1}{2}\,\ddt\,||\varrho||_{L^2}^2 \leq \frac{1}{2}\,||\divx(\tfatu)||_{L^\infty}||\varrho||_{L^2}^2 \leq c||\tfatu||_{H^4}||\varrho||_{H^3}^2\,. \notag
    \end{align}
    By taking the derivatives $\ppi,\ppj\ppi,\ppk\ppj\ppi$ of \eqref{linear_density1}\footnote{The computation where we apply the third order derivative to \eqref{linear_density1} is only formal since we have no information on the existence and regularity of $\ppk\ppj\ppi\varrho_t$ and $\ppk\ppj\ppi\gradx\varrho$. The rigorous proof makes use of a regularization argument.}, testing with $\ppi\varrho,\ppj\ppi\varrho,\ppk\ppj\ppi\varrho$, integrating over $\Omega$ and integrating by parts the terms $\into\tfatu\cdot\gradx(\ppi\varrho)^2\;\dx,\into\tfatu\cdot\gradx(\ppj\ppi\varrho)^2\;\dx,\into\tfatu\cdot\gradx(\ppk\ppj\ppi\varrho)^2\;\dx$, we obtain the same estimates for $\ddt||\ppi\varrho||_{L^2}^2,\ddt||\ppj\ppi\varrho||_{L^2}^2,\ddt||\ppk\ppj\ppi\varrho||_{L^2}^2$, respectively. Consequently,
    \begin{align}
        \ddt\,||\varrho||_{H^3}^2 \leq c||\tfatu||_{H^4}||\varrho||_{H^3}^2\,. \notag
    \end{align}
    Applying Gronwall's inequality, we get
    \begin{align}
        ||\varrho(t)||_{H^3}^2 \leq ||\varrho(0)||_{H^3}^2\exp\!\left(\int_{0}^{T}\!c||\tfatu||_{H^4}\,\ds\right) \notag
    \end{align}
    for a.a. $t\in[0,T]$. Taking the square root on both sides of this inequality we derive
    \eqref{est_sigma}. Estimate \eqref{est_sigmat} follows immediately from \eqref{linear_density1} and \eqref{est_sigma}.
\end{proof}

Analogously, we obtain the following result for the linearized $Z$-equation \eqref{linear_pt1}--\eqref{linear_pt2}.

\begin{lemma}\label{lem_estimates_Z}
    Let $\po\in C^1$, $T>0$, $\tfatu\in L^1(0,T;H^4(\Omega)^d)\cap C([0,T];H^3(\Omega)^d)$ with $\tfatu|_{[0,T]\times\po}=\bm{0}$ and $Z_0\in H^3(\Omega)$. Then there exists a unique solution $Z\in C([0,T];H^3(\Omega))$ to \eqref{linear_pt1}, \eqref{linear_pt2} such that $Z_t\in C([0,T];H^2(\Omega))$. Moreover, if $c_2,c_3$ are as in Lemma \ref{lem_estimates_rho}, then
    \begin{align}
        ||Z||_{L^\infty H^3} &\leq ||Z_0||_{H^3}\exp(c_2||\tfatu||_{L^1H^4})\,, \label{est_omega} \\[2mm]
        ||Z_t||_{L^\infty H^2} &\leq c_3||\tfatu||_{L^\infty H^3}||Z_0||_{H^3}\exp(c_2||\tfatu||_{L^1H^4}) \label{est_omegat}
    \end{align}
    for all $(T,\tfatu,Z_0,Z)$ satisfying \eqref{linear_pt1}, \eqref{linear_pt2} and having the aforementioned properties.
\end{lemma}

\subsection{Existence of local-in-time strong solutions}\label{sec_existence}
Next, we combine the results of the previous subsections with Schauder's fixed-point theorem to prove the existence of local-in-time strong solutions to \eqref{system_1}--\eqref{pressure_with_Z}, \eqref{viscosity_coefficients}. 

\begin{theorem}\label{existence_of_strong_solutions}
    Let $\po\in C^4$, $0<m<M$, 
    $\fatb\in L^2_{\mathrm{loc}}(\reals^+;H^2(\Omega)^d)$ such that $\fatb_t\in L^2_{\mathrm{loc}}(\reals^+;L^2(\Omega)^d)$, $\fatu_0\in H^3(\Omega)^d\cap H^1_0(\Omega)^d$, $\varrho_0, Z_0\in H^3(\Omega)$, $m \leq \varrho_0(\fatx),Z_0(\fatx) \leq M$ for all $\fatx\in\Omega$ and $\gradx p(Z_0)|_{\po} = (\varrho_0\fatb(0) + L\fatu_0)|_{\po}$. Then there exist a time $T^\star\in(0,T]$ and a solution $(\fatu,\varrho,Z)$ to \eqref{system_1}--\eqref{pressure_with_Z}, \eqref{viscosity_coefficients} in $Q_{T^\star}$ with the following properties:
    $\varrho,Z\in C([0,T^\star];H^3(\Omega))$, $\fatu\in L^2(0,T^\star;H^4(\Omega)^d)\cap C([0,T^\star];H^3(\Omega)^d)$, $\fatu_t\in L^2(0,T^\star;H^2(\Omega)^d)$, $\fatu_{tt}\in L^2(0,T^\star;L^2(\Omega)^d)$, $\varrho_t,Z_t\in C([0,T^\star];H^2(\Omega))$ and $\varrho,Z>0$ in $\overline{Q_{T^\star}}$.
\end{theorem}

\begin{proof}
We will prove Lemma \ref{existence_of_strong_solutions} via a fixed-point argument. More specifically, we will apply Schauder's fixed-point theorem (cf. Theorem \ref{schauder}) to the following setup: The underlying Banach space is $\mathbb{X}=\mathbb{X}_{T^\star}=C([0,T^\star];H^2(\Omega)^{d+2})$, where $T^\star\in(0,T]$.
The map $\Phi$ is defined on the set
\begin{align}
    R = R_{T^\star} = \Big\{(\tfatu,\trho,\tZ)\,\Big|\,&\tfatu\in C([0,T^\star];H^3(\Omega)^d)\cap L^2(0,T^\star;H^4(\Omega)^d), \notag \\[2mm]
    &\tfatu_t\in C([0,T^\star];H^1(\Omega)^d)\cap L^2(0,T^\star;H^2(\Omega)^d), \notag \\[2mm]
    &||\tfatu||_{L^\infty H^3}^2 + ||\tfatu||_{L^2 H^4}^2 + ||\tfatu_t||_{L^\infty H^1}^2 + ||\tfatu_t||_{L^2 H^2}^2 \leq B_2, \notag \\[2mm]
    &\tfatu(0) = \fatu_0, \; \tfatu|_{[0,T^\star]\times\po}=\bm{0}, \notag \\[2mm]
    &\trho,\tZ\in C([0,T^\star];H^3(\Omega)), \;\;\; ||\trho||_{L^\infty H^3} + ||\tZ||_{L^\infty H^3} \leq B_1, \notag \\[2mm]
    &\trho_t,\tZ_t\in C([0,T^\star];H^2(\Omega)), \;\;\; ||\trho_t||_{L^\infty H^2} + ||\tZ_t||_{L^\infty H^2} \leq B_3, \notag \\[2mm]
    &(\trho,\tZ)(0) = (\varrho_0, Z_0), \;\;\; \frac{m}{2}\leq \trho,\tZ\leq 2M \; \text{in $\overline{Q_{T^\star}}$}\Big\} \notag
\end{align}
and maps an element $(\tfatu,\trho,\tZ)$ of $R_{T^\star}$ to the triplet $(\fatu,\varrho,Z)$ that contains the unique solutions to \eqref{linear_momentum1}--\eqref{linear_momentum3}, \eqref{linear_density1}--\eqref{linear_density2} and \eqref{linear_pt1}--\eqref{linear_pt2} with $T$ replaced by $T^\star$ and
\begin{gather}
    \fatF \equiv \trho\fatb-\trho(\tfatu\cdot\gradx)\tfatu - \gradx p(\tZ). \notag
\end{gather}
Clearly, every fixed point of $\Phi$ represents a solution to \eqref{system_1}--\eqref{pressure_with_Z}, \eqref{viscosity_coefficients} in $Q_{T^\star}$. To ensure that our setup meets the requirements of Schauder's fixed-point theorem, we need to choose the constants $T^\star$, $B_1$, $B_2$ and $B_3$ appropriately.
First of all, if $B_1$ and $B_2$ are large enough, it turns out that $R_{T^\star}\neq \emptyset$ for every $T^\star\in(0,T]$. Indeed, let $\mathscr{R}$ denote the continuous linear right inverse of the map  
\begin{align}
    \big\{\fatw\in L^2(\reals^+;H^2(\Omega)^d)\,\big|\,\fatw_t\in L^2(\reals^+;L^2(\Omega)^d)\big\} \to H^1(\Omega)^d, \qquad \fatw \mapsto \fatw(0); \notag
\end{align}
see Theorem \ref{Lions_right_inverse}. Then it follows from Corollary \ref{corollary_solutions_linearized_momentum_equation}, Theorem \ref{Lions_right_inverse} and Theorem \ref{Lions_embedding} that the unique solution $\fatu^*$ to
\begin{alignat}{2}
    \fatu_t^* - L\fatu^* &= -\mathscr{R}(L\fatu_0) && \qquad \text{in $Q_\infty$,} \notag \\[2mm]
    \fatu^*|_{[0,\infty)\times\po} &= \bm{0}, && \notag \\[2mm]
    \fatu^*(0,\cdot) &= \fatu_0 && \qquad \text{in $\Omega$} \notag
\end{alignat}
satisfies
\begin{align}
    & ||\fatu^*||_{L^{\infty}(\reals^+;H^3(\Omega)^d)}^2 + ||\fatu^*||_{L^2(\reals^+;H^4(\Omega)^d)}^2 + ||\fatu_t^*||_{L^{\infty}(\reals^+;H^1(\Omega)^d)}^2 + ||\fatu_t^*||_{L^2(\reals^+;H^2(\Omega)^d)}^2
    \notag \\[2mm]
    & \qquad \leq c_1\left(||\mathscr{R}L\fatu_0||_{L^2(\reals^+;H^2(\Omega)^d)}^2 + ||\mathscr{R}L\fatu_0||_{L^\infty(\reals^+;H^1(\Omega)^d)}^2 + ||(\mathscr{R}L\fatu_0)_t||_{L^2(\reals^+;L^2(\Omega)^d)}^2\right) \notag \\[2mm]
    & \qquad \leq c_1(1+C_5)\left(||\mathscr{R}L\fatu_0||_{L^2(\reals^+;H^2(\Omega)^d)}^2 + ||(\mathscr{R}L\fatu_0)_t||_{L^2(\reals^+;L^2(\Omega)^d)}^2\right) \notag \\[2mm]
    & \qquad \leq c_1(1+C_5)C_6||L\fatu_0||_{H^1(\Omega)^d}^2. \notag
\end{align}
Thus, if we take 
\begin{align}
    B_2 > c_1(1+C_5)C_6||L\fatu_0||_{H^1(\Omega)^d}^2 \qquad \text{and} \qquad B_1 > 2(||\varrho_0||_{H^3} + ||Z_0||_{H^3}), \label{B1B2_bound_1}
\end{align}
then $(\fatu^*,\varrho_0,Z_0)\in R_{T^\star}$ for all $T^\star\in(0,T]$. Henceforth, $B_1$ and $B_2$ satisfy \eqref{B1B2_bound_1}. Next, we ensure that $\Phi(R_{T^\star})\subset R_{T^\star}$. For this purpose, let $(\tfatu,\trho,\tZ)\in R_{T^\star}$ be arbitrary and $(\fatu,\varrho,Z) = \Phi(\tfatu,\trho,\tZ)$. We need the following estimates that follow easily from the definitions of $R_{T^\star}$ and $\fatF$:
\begin{align}
    ||\fatF||_{L^2H^2}^2 &\leq c(B_1,B_2)\left\{\big(1+||p||_{C^3}^2\big)T^\star + ||\fatb||_{L^2H^2}^2\right\}, \notag \\[2mm]
    ||\fatF_t||_{L^2L^2}^2 &\leq c(B_1,B_2,B_3)\left\{\big(1+||p||_{C^2}^2\big)T^\star + ||\fatb||_{L^2L^2}^2 + ||\fatb_t||_{L^2L^2}^2\right\}, \notag \\[2mm]
    ||\fatF_t||_{L^2H^1}^2 &\leq c(B_1,B_2,B_3)\left\{\big(1+||p||_{C^3}^2\big)T^\star + 1 + ||\fatb||_{L^2H^1}^2 + ||\fatb_t||_{L^2H^1}^2\right\}, \notag \\[2mm]
    ||\fatF||_{L^\infty H^1}^2 &\leq 4\left(||\fatF(0)||_{H^1}^2 + T^\star||\fatF_t||_{L^2H^1}^2\right), \notag \\[2mm]
    ||\trho_t||_{L^2L^3}^2 &\leq c(B_3) T^\star. \notag
\end{align}
Furthermore, it follows from \eqref{est_sigma}, \eqref{est_sigmat}, \eqref{est_omega}, \eqref{est_omegat} and \eqref{B1B2_bound_1} that
\begin{align}
    ||\varrho||_{L^\infty H^3} + ||Z||_{L^\infty H^3} &\leq \frac{1}{2}B_1\exp[c(B_2)(T^\star)^{1/2}], \notag \\[2mm]
    ||\varrho_t||_{L^\infty H^2} + ||Z_t||_{L^\infty H^2} &\leq c_3B_2^{1/2}B_1\exp[c(B_2)(T^\star)^{1/2}]. \notag 
\end{align}
In addition, \eqref{u_problem_2} yields
\begin{align}
    &||\fatu||_{L^\infty H^3}^2 + ||\fatu||_{L^2H^4}^2 + ||\fatu_t||_{L^\infty H^1}^2 + ||\fatu_t||_{L^2 H^2}^2 \notag \\[2mm]
    &\leq c_1\left\{||\fatF||_{L^2 H^2}^2 + ||\fatF||_{L^\infty H^1}^2
    + \left(||\fatF_t||_{L^2 L^2}^2 + ||(\fatu_t)_0||_{H^1}^2\right)\right. \notag \\[2mm]
    &\qquad\left.{}\cdot(1+||\trho_t||_{L^2 L^3}^2)(1 + ||\gradx \trho||_{L^\infty H^1}^2)\exp(c_1||\trho_t||_{L^2 L^3}^2)\right\} \notag \\[2mm]    
    &\leq c_1\left(||\fatF||_{L^2H^2}^2 + 4\left(||\fatF(0)||_{H^1}^2 + T^\star||\fatF_t||_{L^2 H^1}^2\right)\right) \notag \\[2mm]
    &\phantom{\leq\;\;} + c_1\left(||\fatF_t||_{L^2 L^2}^2 + ||(\fatu_t)_0||_{H^1}^2\right)
    (1+c(B_3)T^\star)(1+B_1^2)\exp(c_1c(B_3)T^\star), \notag
\end{align}
where
\begin{align}
    (\fatu_t)_0 \equiv \frac{\fatF_0 + L\fatu_0}{\varrho_0} \in H^1_0(\Omega)^d, \qquad \fatF_0 &\equiv \fatF(0) = \varrho_0\fatb(0) - (\fatu_0\cdot\gradx)\fatu_0 -\gradx p(Z_0)\in H^1(\Omega). \notag
\end{align}
Hence, if we take
\begin{align}
    B_2 > 4c_1||\fatF_0||_{H^1}^2 + c_1(1+B_1^2)||(\fatu_t)_0||_{H^1}^2 \qquad \text{and} \qquad  
    B_3 > c_3B_2^{1/2}B_1 \notag
\end{align}
and $T^\star$ small enough, $\fatu,\varrho,Z$ satisfy all estimates listed in the definition of $R_{T^\star}$ with the exception of the lower and upper bounds for $\varrho$ and $Z$.
However, since $H^2(\Omega)\hookrightarrow C(\overline{\Omega})$ for both $d=2$ and $d=3$, 
\begin{align}
    ||\varrho-\varrho_0||_{C(\overline{Q_{T^\star}})} + ||Z-Z_0||_{C(\overline{Q_{T^\star}})} &\leq c\left(||\varrho-\varrho_0||_{L^\infty H^{2}} + ||Z-Z_0||_{L^\infty H^{2}}\right) \notag \\[2mm]
    &\leq c{T^\star}\left(||\varrho_t||_{L^\infty H^2} + ||Z_t||_{L^\infty H^2}\right)\leq c{T^\star}B_3. \notag
\end{align}
Consequently, if ${T^\star}$ is sufficiently small, also the lower and upper bounds for $\varrho$ and $Z$ are satisfied. We have thus shown that $\Phi(R_{T^\star}) \subset R_{T^\star}$ provided ${T^\star}$ is small enough. We henceforth assume that $T^\star$ is chosen accordingly. Next, we observe that $R_{T^\star}$ is a convex and closed subset of $\mathbb{X}_{T^\star}$.
Moreover, since $R_{T^\star}$ is bounded in 
\begin{align}
    \big\{(\tfatu,\trho,\tZ)\in L^\infty(0,T^\star;H^3(\Omega)^{d+2})\,\big|\, (\tfatu,\trho,\tZ)_t\in L^2(0,T^\star;H^1(\Omega)^{d+2})\big\}, \notag
\end{align} 
it follows from the Aubin-Lions lemma (cf. Theorem \ref{Aubin-Lions_lemma} in the appendix) that it is precompact in $\mathbb{X}_{T^\star}$. As $R_{T^\star}$ is precompact and closed in $\mathbb{X}_{T^\star}$, it is compact in $\mathbb{X}_{T^\star}$. Thus, it remains to show that $\Phi$ is continuous. To this end, let $\{(\tfatu_n,\trho_n,\tZ_n)\}_{n\,\in\,\mathbb{N}}\subset R_{T^\star}$ be a sequence that converges in $\mathbb{X}_{T^\star}$ to $(\tfatu,\trho,\tZ)\in R_{T^\star}$. We set $(\fatu_n,\varrho_n,Z_n) = \Phi(\tfatu_n,\trho_n,\tZ_n)$, $(\fatu,\varrho,Z) = \Phi(\tfatu,\trho,\tZ)$ and observe that these quantities satisfy the equations
\begin{align}
    \trho(\fatu-\fatu_n)_t &+ \frac{1}{2}\,\trho_t(\fatu-\fatu_n) - L(\fatu-\fatu_n) \notag \\[2mm]
    &= (\trho-\trho_n)\fatb - (\trho-\trho_n)(\fatu_n)_t + \frac{1}{2}\,\trho_t(\fatu-\fatu_n) - \gradx(p(\tZ)-p(\tZ_n)) \notag \\[2mm]
    &\phantom{=\;} - (\trho-\trho_n)\gradx\tfatu\cdot\tfatu - \trho_n\gradx\tfatu\cdot(\tfatu-\tfatu_n) - \trho_n\gradx(\tfatu-\tfatu_n)\cdot\tfatu_n\,, \label{L_u_1} \\[2mm]
    (\varrho-\varrho_n)_t &= \tfatu\cdot\gradx(\varrho_n-\varrho) + (\tfatu_n-\tfatu)\cdot\gradx\varrho_n + \varrho_n\,\divx(\tfatu_n-\tfatu) + (\varrho_n-\varrho)\,\divx(\tfatu)\,, \label{L_sigma_1} \\[2mm]
    (Z-Z_n)_t &= \tfatu\cdot\gradx(Z_n-Z) + (\tfatu_n-\tfatu)\cdot\gradx Z_n + Z_n\,\divx(\tfatu_n-\tfatu) + (Z_n-Z)\,\divx(\tfatu)\,. \label{L_omega_1}
\end{align}
We multiply \eqref{L_u_1} by $\fatu-\fatu_n$ and integrate over $\Omega$. Integrating by parts the term on the left-hand side involving the operator $L$, we obtain 
\begin{align}
    &\frac{1}{2}\,\ddt\into\trho\,|\fatu-\fatu_n|^2\;\dx + \mu\into|\gradx(\fatu-\fatu_n)|^2\;\dx + (\lambda+\mu)\into|\divx(\fatu-\fatu_n)|^2\;\dx \notag \\[2mm]
    &= \into(\trho-\trho_n)(\fatu-\fatu_n)\cdot\fatb\;\dx + \frac{1}{2}\into\trho_t|\fatu-\fatu_n|^2\;\dx + \into [p(\tZ)-p(\tZ_n)]\,\divx(\fatu-\fatu_n)\;\dx \notag \\[2mm]
    &\phantom{=\;} - \into(\trho-\trho_n)(\fatu-\fatu_n)^T\cdot\gradx\tfatu\cdot\tfatu\;\dx - \into\trho_n(\fatu - \fatu_n)^T\cdot\gradx\tfatu\cdot(\tfatu-\tfatu_n)\;\dx \notag \\[2mm]
    &\phantom{=\;}  - \into\trho_n(\fatu-\fatu_n)^T\cdot\gradx(\tfatu-\tfatu_n)\cdot\tfatu_n\;\dx \notag \\[2mm] 
    &\leq ||\fatb||_{L^2}||\fatu-\fatu_n||_{L^\infty}||\trho-\trho_n||_{L^2} + ||\trho-\trho_n||_{L^2}||(\fatu_n)_t||_{L^2}||\fatu-\fatu_n||_{L^\infty} + ||\trho_t||_{L^\infty}||\fatu-\fatu_n||_{L^2}^2 \notag \\[2mm]
    &\phantom{=\;} + ||p||_{C^1}||\tZ-\tZ_n||_{L^2}||\divx(\fatu-\fatu_n)||_{L^2} + ||\trho-\trho_n||_{L^2}||\gradx\tfatu||_{L^2}||\tfatu||_{L^\infty}||\fatu-\fatu_n||_{L^\infty} \notag \\[2mm]
    &\phantom{=\;} + ||\trho_n||_{L^\infty}||\gradx\tfatu||_{L^2}||\tfatu-\tfatu_n||_{L^2}||\fatu-\fatu_n||_{L^\infty} + ||\trho_n||_{L^\infty}||\gradx(\tfatu-\tfatu_n)||_{L^2}||\tfatu_n||_{L^\infty}||\fatu-\fatu_n||_{L^2}. \label{L_u_2}
\end{align}
Integrating \eqref{L_u_2} in time and using the bounds provided by $R_{T^\star}$, we obtain
\begin{align}
    ||\fatu(t)-\fatu_n(t)||_{L^2}^2 &\leq c(B_1,B_2,B_3)\Big(T^\star(1+||\fatb||_{L^\infty L^2})||\trho-\trho_n||_{L^\infty L^2} + T^\star||p||_{C^1}||\tZ-\tZ_n||_{L^\infty L^2} \notag \\[2mm]
    &\phantom{\leq c(B_1,B_2,B_3)\Big(} + T^\star||\tfatu-\tfatu_n||_{L^\infty H^1} + \int_0^{\,t}||\fatu-\fatu_n||_{L^2}^2\;\ds\Big) \notag
\end{align}
for a.a. $t\in(0,T^\star)$. Using Gronwall's lemma and $(\tfatu_n,\trho_n,\tZ_n)\to(\tfatu,\trho,\tZ)$ in $\mathbb{X}_{T^\star}$ as $n\to\infty$, we deduce that $\fatu_n\to\fatu$ in $C([0,T^\star];L^2(\Omega)^d)$ as $n\to\infty$. Next, we multiply \eqref{L_sigma_1} by $\varrho-\varrho_n$ and integrate over $\Omega$. Integrating by parts the first term on the right-hand side of the resulting equation, we obtain
\begin{align}
    \frac{1}{2}\,\ddt\into|\varrho-\varrho_n|^2\;\dx &= -\frac{1}{2}\into\divx(\tfatu)|\varrho-\varrho_n|^2\;\dx + \into(\varrho-\varrho_n)(\tfatu_n-\tfatu)\cdot\gradx\varrho_n\;\dx \notag \\[2mm]
    &\phantom{=\;} + \into\varrho_n(\varrho-\varrho_n)\divx(\tfatu_n-\tfatu)\;\dx \notag \\[2mm]
    &\leq ||\divx(\tfatu)||_{L^\infty}||\varrho-\varrho_n||_{L^2}^2 + ||\varrho-\varrho_n||_{L^\infty}||\tfatu-\tfatu_n||_{L^2}||\gradx\varrho_n||_{L^2} \notag \\[2mm]
    &\phantom{\leq\;} + ||\varrho_n||_{L^\infty}||\varrho-\varrho_n||_{L^2}||\divx(\tfatu-\tfatu_n)||_{L^2}\,. \label{L_sigma_2}
\end{align}
Integrating \eqref{L_sigma_2} in time and using the bounds provided by $R_{T^\star}$, we obtain
\begin{align}
    ||\varrho(t)-\varrho_n(t)||_{L^2}^2 \leq c(B_1,B_2)\left(T^\star||\tfatu-\tfatu_n||_{L^\infty H^1} + \int_0^{\,t} ||\varrho-\varrho_n||_{L^2}^2\;\ds\right) \notag
\end{align}
for a.a. $t\in(0,T^\star)$. Using Gronwall's lemma and $(\tfatu_n,\trho_n,\tZ_n)\to(\tfatu,\trho,\tZ)$ in $\mathbb{X}_{T^\star}$ as $n\to\infty$, we deduce that $\varrho_n\to\varrho$ in $C([0,T^\star];L^2(\Omega))$ as $n\to\infty$. Analogously, we show that $Z_n\to Z$ in $C([0,T^\star];L^2(\Omega))$ as $n\to\infty$. Thus,
$(\fatu_n,\varrho_n,Z_n)\to(\fatu,\varrho,Z)$ in
$C([0,T^\star];L^2(\Omega)^{d+2})$ as $n\to\infty$.
Since $R_{T^\star}$ is compact, it follows that $(\fatu_n,\varrho_n,Z_n)\to(\fatu,\varrho,Z)$ in $\mathbb{X}_{T^\star}$ as $n\to\infty$. Thus, we have shown that $\Phi$ is continuous. In particular, Schauder's fixed-point theorem implies that $\Phi$ has a fixed point, which is a strong solution to \eqref{system_1}--\eqref{pressure_with_Z}, \eqref{viscosity_coefficients} in $Q_{T^\star}$ with the desired properties.
\end{proof}

\subsection{Uniqueness of strong solutions}\label{sec_uniqueness}
In this section, we prove that the solution $(\fatu,\varrho,Z)$ to \eqref{system_1}--\eqref{pressure_with_Z}, \eqref{viscosity_coefficients} is unique in the class 
\begin{gather}
    \left.
    \begin{array}{c}
        (\fatu,\varrho,Z)\in L^\infty(Q_T)^{d+2}, \quad\;\;\; \divx(\fatu)\in L^1(0,T;L^\infty(\Omega)), \quad\;\;\; \gradx\fatu\in L^2(0,T;L^3(\Omega)^{d\times d}), \\[4mm]
        \underset{(t,\fatx)\,\in\,Q_T}{\mathrm{ess\,inf}}\{\varrho(t,\fatx)\},\underset{(t,\fatx)\,\in\,Q_T}{\mathrm{ess\,inf}}\{Z(t,\fatx)\} > 0, \quad\;\;\; \fatu_t,\gradx\varrho,\gradx Z\in L^2(0,T;L^3(\Omega)^d),
    \end{array} 
    \right\} \label{strong_solution_class}
\end{gather}
provided $\po\in C^1$
and $\fatb\in L^2(0,T;L^3(\Omega)^d)$. To see this, we suppose that there are two (strong) solutions $(\fatu_1,\varrho_1,Z_1)$ and $(\fatu_2,\varrho_2,Z_2)$ in the class \eqref{strong_solution_class}. In this case, the functions $\fatv = \fatu_1-\fatu_2$, $\eta = \varrho_1-\varrho_2$, $\zeta = Z_1-Z_2$ satisfy the following equations in $Q_T$:
\begin{align}
    \varrho_1\fatv_t - L\fatv &= \eta[\fatb - (\fatu_2)_t - (\fatu_2\cdot\gradx)\fatu_2] - \varrho_1[(\fatu_1\cdot\gradx)\fatv + (\fatv\cdot\gradx)\fatu_2] - \gradx(p(Z_1)-p(Z_2))\,, \label{uniqueness_eq_1} \\[2mm]
    \eta_t &= - \fatu_1\cdot\gradx\eta - \fatv\cdot\gradx\varrho_2 - \varrho_1\divx(\fatv) - \eta\,\divx(\fatu_2)\,, \label{uniqueness_eq_2} \\[2mm]
    \zeta_t &= -\fatu_1\cdot\gradx\zeta - \fatv\cdot\gradx Z_2 - Z_1\divx(\fatv) - \zeta\,\divx(\fatu_2) \,. \label{uniqueness_eq_3}
\end{align}
We test \eqref{uniqueness_eq_1} with $\fatv$ and integrate over $\Omega$. Integrating by parts some terms and observing that, by \eqref{system_1},
\begin{align}
    \frac{1}{2}\into (\varrho_1)_t|\fatv|^2\;\dx = - \frac{1}{2}\into\divx(\varrho_1\fatu_1)|\fatv|^2\;\dx = \into \varrho_1\fatv^T\cdot\gradx\fatv\cdot\fatu_1\;\dx\,, \notag
\end{align}
we obtain
\begin{align}
    &\frac{1}{2}\,\ddt\into \varrho_1|\fatv|^2\;\dx + \mu\into |\gradx\fatv|^2\;\dx + (\lambda+\mu)\into |\divx(\fatv)|^2\;\dx \notag \\[2mm]
    &= \into \eta[\fatb-(\fatu_2)_t]\cdot\fatv\;\dx - \into \eta\fatv^T\cdot\gradx\fatu_2\cdot\fatu_2\;\dx - \into \varrho_1\fatv^T\cdot\gradx\fatu_2\cdot\fatv\;\dx + \into[p(Z_1)-p(Z_2)]\,\divx(\fatv)\;\dx\notag \\[2mm]
    &\leq ||\eta||_{L^2}||\fatv||_{L^6}\big(||\fatb||_{L^3} + ||(\fatu_2)_t||_{L^3} + ||\gradx\fatu_2||_{L^3}||\fatu_2||_{L^\infty}\big) + ||\varrho_1||_{L^\infty}||\fatv||_{L^2}||\gradx\fatu_2||_{L^3}||\fatv||_{L^6} \notag \\[2mm]
    &\phantom{\leq\;} + ||p||_{C^1}||\zeta||_{L^2}||\divx(\fatv)||_{L^2} \notag \\[2mm]
    &\leq \frac{4\varepsilon}{6}\,||\fatv||_{L^6}^2 + \frac{3}{2\varepsilon}\,||\eta||_{L^2}^2\big(||\fatb||_{L^3}^2 + ||(\fatu_2)_t||_{L^3}^2 + ||\gradx\fatu_2||_{L^3}^2||\fatu_2||_{L^\infty}^2\big) \notag \\[2mm]
    &\phantom{\leq\;} + \frac{3}{2\varepsilon}\,||\varrho_1||_{L^\infty}^2||\fatv||_{L^2}^2||\gradx\fatu_2||_{L^3}^2 + \frac{3}{4\delta}\,||p||_{C^1}^2||\zeta||_{L^2}^2 + \frac{\delta}{3}\,||\divx(\fatv)||_{L^2}^2 \notag \\[2mm]
    &\leq c\left\{\frac{4\varepsilon}{6}\into|\gradx\fatv|^2\;\dx + \frac{\delta}{3}\into|\divx(\fatv)|^2\;\dx + \left(\frac{3}{2\varepsilon}\Big[||\fatb||_{L^3}^2 + ||(\fatu_2)_t||_{L^3}^2 + ||\gradx\fatu_2||_{L^3}^2||\fatu_2||_{L^\infty}^2 \right.\right. \notag \\[2mm]
    &\phantom{\leq\;c} \left.\left. {}+ ||\varrho_1||_{L^\infty}^2||\gradx\fatu_2||_{L^3}^2\Big] + \frac{3}{4\delta}\,||p||_{C^1}^2\right)\into \big(\varrho_1|\fatv|^2 + \varrho_1|\eta|^2 + Z_1|\zeta|^2\big)\,\dx\right\} \label{uniqueness_eq_4}
\end{align}
for all $\varepsilon,\delta>0$. Next, we test \eqref{uniqueness_eq_2} with $\varrho_1\eta$ and integrate over $\Omega$. Together with the observation that, by \eqref{system_1},
\begin{align}
    \frac{1}{2}\into (\varrho_1)_t|\eta|^2\;\dx = - \frac{1}{2}\into\divx(\varrho_1\fatu_1)|\eta|^2\;\dx = \into \varrho_1\eta\fatu_1\cdot\gradx\eta\;\dx\,, \notag
\end{align}
we obtain
\begin{align}
    &\frac{1}{2}\,\ddt\into \varrho_1|\eta|^2\;\dx = -\into\varrho_1\eta\fatv\cdot\gradx\varrho_2\;\dx - \into\varrho_1^2\eta\,\divx(\fatv)\;\dx - \into\varrho_1|\eta|^2\divx(\fatu_2)\;\dx \notag \\[2mm]
    &\qquad\quad\leq ||\varrho_1||_{L^\infty}||\eta||_{L^2}||\fatv||_{L^6}||\gradx\varrho_2||_{L^3} + ||\varrho_1||_{L^\infty}^2||\eta||_{L^2}||\divx(\fatv)||_{L^2} 
    + ||\divx(\fatu_2)||_{L^\infty}\!\into\varrho_1|\eta|^2\;\dx \notag \\[2mm]
    &\qquad\quad\leq \frac{\varepsilon}{6}\,||\fatv||_{L^6}^2 + \frac{3}{2\varepsilon}\,||\varrho_1||_{L^\infty}^2||\eta||_{L^2}^2||\gradx\varrho_2||_{L^3}^2 + \frac{\delta}{3}\,||\divx(\fatv)||_{L^2}^2 + \frac{3}{4\delta}\,||\varrho_1||_{L^\infty}^4||\eta||_{L^2}^2 
    \notag \\[2mm]
    &\qquad\quad\phantom{\leq\;} 
    + ||\divx(\fatu_2)||_{L^\infty}\!\into\varrho_1|\eta|^2\;\dx \notag \\[2mm]
    &\qquad\quad\leq c\left\{\left[\frac{3}{2\varepsilon}\,||\varrho_1||_{L^\infty}^2||\gradx\varrho_2||_{L^3}^2  + \frac{3}{4\delta}\,||\varrho_1||_{L^\infty}^4 + ||\divx(\fatu_2)||_{L^\infty}\right]\into\varrho_1|\eta|^2\;\dx
    \right. \notag \\[2mm]
    &\qquad\quad\phantom{\leq\;c}\left. {} 
    + \frac{\varepsilon}{6}\into|\gradx\fatv|^2\;\dx + \frac{\delta}{3}\into|\divx(\fatv)|^2\;\dx\right\} \label{uniqueness_eq_5}
\end{align}
for all $\varepsilon,\delta>0$. Analogously, we get
\begin{align}
    \frac{1}{2}\,\ddt\into Z_1|\zeta|^2\;\dx &\leq c\left\{\left[\frac{3}{2\varepsilon}\,||Z_1||_{L^\infty}^2||\gradx Z_2||_{L^3}^2  + \frac{3}{4\delta}\,||Z_1||_{L^\infty}^4 + ||\divx(\fatu_2)||_{L^\infty}\right]\into Z_1|\zeta|^2\;\dx \right. \notag \\[2mm]
    &\phantom{=\;c}\left.{}+ \frac{\varepsilon}{6}\into|\gradx\fatv|^2\;\dx + \frac{\delta}{3}\into|\divx(\fatv)|^2\;\dx\right\} \label{uniqueness_eq_6}
\end{align}
for all $\varepsilon,\delta>0$.
Summing up \eqref{uniqueness_eq_4}--\eqref{uniqueness_eq_6}, we get
\begin{align}
    &\frac{1}{2}\,\ddt\into \big(\varrho_1|\fatv|^2 + \varrho_1|\eta|^2 + Z_1|\zeta|^2\big)\,\dx + \mu\into |\gradx\fatv|^2\;\dx + (\lambda+\mu)\into |\divx(\fatv)|^2\;\dx \notag \\[2mm]
    &\quad \leq c\left\{\varepsilon\into |\gradx\fatv|^2\;\dx + \delta\into|\divx(\fatv)|^2\;\dx + \left(\frac{3}{2\varepsilon}\left[||\varrho_1||_{L^\infty}^2||\gradx\varrho_2||_{L^3}^2   + ||Z_1||_{L^\infty}^2||\gradx Z_2||_{L^3}^2  \right.\right.\right. \notag \\[2mm]
    &\quad \phantom{=\;c} \left.{}+ ||\fatb||_{L^3}^2 + ||(\fatu_2)_t||_{L^3}^2 + ||\gradx\fatu_2||_{L^3}^2||\fatu_2||_{L^\infty}^2 + ||\varrho_1||_{L^\infty}^2||\gradx\fatu_2||_{L^3}^2\right] + ||\divx(\fatu_2)||_{L^\infty} \notag \\[2mm]
    &\quad \phantom{=\;c}\left.\left. {} + \frac{3}{4\delta}\left[||\varrho_1||_{L^\infty}^4 + ||Z_1||_{L^\infty}^4 + ||p||_{C^1}^2\right]\right)\into \big(\varrho_1|\fatv|^2 + \varrho_1|\eta|^2 + Z_1|\zeta|^2\big)\,\dx \right\} \notag
\end{align}
for all $\varepsilon,\delta>0$. Choosing $\varepsilon=\mu/(2c)$ and $\delta=(\lambda+\mu)/(2c)$, we may absorb the first two terms on the right-hand side into the last two terms on the left-hand side. Finally, applying Gronwall's lemma, we deduce that $(\fatu_1,\varrho_1,Z_1)=(\fatu_2,\varrho_2,Z_2)$. Thus, there exists at most one strong solution in the class \eqref{strong_solution_class}.

\section{Conditional regularity}\label{sec_conditional_regularity}
Since the results of Section \ref{sec_local_in_time_strong_solutions} only guarantee that the strong solution to \eqref{system_1}--\eqref{pressure_with_Z}, \eqref{viscosity_coefficients} exists locally in time, it is desirable to investigate under which conditions it exists globally in time.
In the literature several criteria for the regularity of solutions to both the incompressible as well as the compressible Navier-Stokes equations were investigated. 
In this paper, we follow the strategy used in \citep{SWZ} to obtain conditional regularity for the barotropic Navier-Stokes equations. In Section \ref{sec_Nash}, we first recall Nash's conjecture formulated in the context of compressible, viscous and heat conducting fluid flows. We also mention the state of the art for results in this direction and state the main result of Section \ref{sec_conditional_regularity} in Theorem \ref{blow_up_criterion}. In Section \ref{sec_preliminaries}, we present some preliminary results that are needed in the proof Theorem \ref{blow_up_criterion} which will be carried out in Section \ref{sec_proof}.

\subsection{Nash's conjecture and blow-up criteria}\label{sec_Nash}
In 1958, Nash \citep[p.\,933]{Nash} formulated the following conjecture for the governing equations of viscous, compressible and heat conducting fluids:
\begin{center}
    \begin{minipage}{0.8\textwidth}
        ``\textit{Probably one should first try to prove a conditional existence and uniqueness theorem for the flow equations. This should give existence, smoothness, and unique continuation (in time) of flows, conditional on the non-appearance of certain gross types of singularity, such as infinities of temperature or density.}''
    \end{minipage}
\end{center}
This conjecture has been proven only recently in \citep{Feireisl_Wen_Zhu}. More precisely, if $(\fatu,\varrho,\vartheta)$ is the strong solution to the Navier-Stokes-Fourier system\footnote{Here, $\vartheta$ denotes the \textit{absolute temperature} of the fluid.} with maximal existence time $\Tmax<\infty$, then it follows from \citep[Theorem 1.4]{Feireisl_Wen_Zhu} that
\begin{align}
    \limsup_{t\,\uparrow\,\Tmax}\,\left(||\varrho(t)||_{L^\infty} + ||\vartheta(t)||_{L^\infty}\right) = \infty. \notag
\end{align}
In the context of the barotropic Navier-Stokes equations one would expect to obtain the same result without the temperature. However, in this case, the best result known to the authors is the following; see \citep{SWZ} with \citep[Lemma 3.1]{SWZ} replaced by \citep[Lemma 3.1]{Wen_Zhu}: 
\begin{enumerate}
    \item[(i)]{If $3\lambda<29\mu$ and $\Tmax<\infty$, then ${\displaystyle\limsup_{t\,\uparrow\,\Tmax}\,\left(||\varrho(t)||_{L^\infty}\right) = \infty}$.}
\end{enumerate}
From the considerations in \citep{SWZ} it further follows that the restriction on the viscosity coefficients $\lambda,\mu$ in the above blow-up criterion can be removed, provided the $L^\infty$-norm of the velocity is incorporated:  
\begin{enumerate}
    \item[(ii)]{If $\Tmax<\infty$, then ${\displaystyle\limsup_{t\,\uparrow\,\Tmax}\,\left(||\varrho(t)||_{L^\infty} + ||\fatu(t)||_{L^\infty}\right) = \infty}$.}
\end{enumerate}
Our aim is to prove the latter
blow-up criterion for the strong solution to \eqref{system_1}--\eqref{pressure_with_Z}, \eqref{viscosity_coefficients} obtained in the previous section.
\begin{theorem}\label{blow_up_criterion}
    Let the assumptions of Theorem \ref{existence_of_strong_solutions} be satisfied and let $(\fatu,\varrho,Z)$ be the strong solution to \eqref{system_1}--\eqref{pressure_with_Z}, \eqref{viscosity_coefficients} with $\fatb=\bm{0}$ and maximal existence time $\Tmax$. If $\Tmax<\infty$, then
    \begin{align}
        \limsup_{t\,\uparrow\,\Tmax}\,\left(||\varrho(t)||_{L^\infty} + ||\fatu(t)||_{L^\infty}\right) = \infty. \notag
    \end{align}
\end{theorem}

\subsection{Preliminary results}\label{sec_preliminaries}
We start by recording some refined results concerning the Lamé operator. 

\begin{lemma}\label{lame_regularity}
    Let $q\in(1,\infty)$ and let $\Omega\subset\reals^d$ be a bounded domain with $\po\in C^2$. \hypertarget{i}{}
    \begin{enumerate}
        \item
        [$\mathrm{(i)}$]
        {The Lamé operator is an isomorphism $W^{2,q}(\Omega)^d\cap W^{1,q}_0(\Omega)^d\to L^q(\Omega)^d$. In particular, there exists a constant $C_0>0$ such that
        \begin{align}
            ||L^{-1}\fatF||_{W^{2,q}} \leq C_0||\fatF||_{L^q} \hypertarget{ii}{} \notag
        \end{align}
        for all $\fatF\in L^q(\Omega)^d$.}
        \item
        [$\mathrm{(ii)}$]
        {There exists a constant $C_1>0$ such that
        \begin{align}
            ||L^{-1}[\gradx f]||_{W^{1,q}} \leq C_1||f||_{L^q} \hypertarget{iii}{} \notag
        \end{align}
        for all $f\in W^{1,q}(\Omega)^d$.}
        \item
        [$\mathrm{(iii)}$]
        {There exists a constant $C_2>0$ such that
        \begin{align}
            ||L^{-1}[\gradx\divx(\fatf)]||_{L^q} \leq C_2||\fatf||_{L^q} \notag
        \end{align}
        for all $\fatf\in W^{2,q}(\Omega)^d\cap W^{1,q}_0(\Omega)^d$.}
    \end{enumerate}
\end{lemma}

\begin{proof}
    For the proof of part (\hyperlink{i}{i}) we refer to \citep[Theorem 5.1]{Shi_Wright}\footnote{Note that in \citep{Shi_Wright} only the case $d=3$ is considered. However, as observed in the proof of \citep[Proposition A.1]{Danchin}, a similar proof applies in the case $d=2$.}. See also \citep[Theorem 10.5]{ADN} for a proof of a priori estimates for general elliptic systems. In the case $\po\in C^3$, the unique solvability of the no-slip boundary value problem for the Lamé system \eqref{lame0}--\eqref{lame01} can also be deduced from Lemma \ref{elliptic_regularity} and \citep[Corollario 1.1 and Teorema 3.5]{Geymonat}.
    
    For a proof of part (\hyperlink{ii}{ii}) see \citep[Theorem 8.1.1 and Section 8.6]{Mayboroda} or \citep[Theorems 3.29 and 3.31]{Ambrosio}. 
    To prove part (\hyperlink{iii}{iii}), let $p = q/(q-1)$ and set $\fatU\equiv L^{-1}[\gradx\divx(\fatf)]$. Then \eqref{lame0}--\eqref{lame01} are satisfied. We multiply \eqref{lame0} by an arbitrary function $\fatphi\in W\equiv W^{2,p}(\Omega)^d\cap W^{1,p}_0(\Omega)^d$. After integrating by parts both sides of the result and taking the absolute value, we obtain
    \begin{align}
        \left|\into\fatU\cdot L\fatphi\;\dx\right| = \left|\into\fatf\cdot\gradx\divx(\fatphi)\;\dx\right| \leq ||\fatf||_{L^q}||\gradx\divx(\fatphi)||_{L^{p}}, \notag
    \end{align}
    where, by part (\hyperlink{i}{i}),
    \begin{align}
        ||\gradx\divx(\fatphi)||_{L^{p}} \leq d\,||\fatphi||_{W^{2,p}} \leq dC_0||L\fatphi||_{L^{p}}. \notag
    \end{align}
    Moreover, it follows from (\hyperlink{i}{i}) that
    \begin{align}
        ||\fatU||_{L^{q}} = \sup_{\fatpsi\,\in\,L^{p}\backslash\{\bm{0}\}}\left\{\left|\into \fatU\cdot\frac{\fatpsi}{||\fatpsi||_{L^p}}\;\dx\right|\right\} = \sup_{\fatphi\,\in\,W\backslash\{\bm{0}\}}\left\{\left|\into \fatU\cdot \frac{L\fatphi}{||L\fatphi||_{L^p}}\;\dx\right|\right\}. \notag
    \end{align}
    Combining the above information, we get (\hyperlink{iii}{iii}).
\end{proof}

Before we can state the next result concerning the Lamé operator, we need to introduce the space of functions $\Omega\to\reals$ with bounded mean oscillation. It is denoted by $\BMO(\Omega)$ and is given by
\begin{align}
    \BMO(\Omega) = \left\{f\in L^2(\Omega)\left|||f||_{\BMO(\Omega)} < \infty\right.\right\}, \notag
\end{align}
where
\begin{align}
    ||f||_{\BMO(\Omega)} &= ||f||_{L^2(\Omega)} + [f]_{\BMO(\Omega)}, \notag \\[2mm]
    [f]_{\BMO(\Omega)} &= \sup_{(\fatx,r)\,\in\,\Omega\times\reals^+}\fint_{\Omega_r(\fatx)}|f(\faty)-f_{\Omega_r(\fatx)}|\;\dy. \notag
\end{align}
Here, $\Omega_r(\fatx)=B_r(\fatx)\cap\Omega$, $B_r(\fatx)$ denotes the ball with center $\fatx$ and radius $r$ and $|\Omega_r(\fatx)|$ is the Lebesgue measure of $\Omega_r(\fatx)$. Noting that
\begin{align}
    [f]_{\BMO} \leq 2||f||_{L^\infty} \notag
\end{align}
for all $f\in L^\infty(\Omega)$, the following lemma is a consequence of \citep[Theorem 2.2]{Acquistapace}.

\begin{lemma}\label{gradu_regularity}
    Let $\Omega\subset\reals^d$ be a bounded domain with $\po\in C^2$. Then $\gradx L^{-1}[\divx(\matrixF)]\in \BMO(\Omega)^{d\times d}$ for all $\matrixF\in W^{1,\infty}(\Omega)^{d\times d}$. In particular, there exists a constant $C_3>0$ such that 
    \begin{align}
        ||\gradx L^{-1}[\divx(\matrixF)]||_{\BMO} \leq C_3||\matrixF||_{L^{\infty}} \notag
    \end{align}
    for all $\matrixF\in W^{1,\infty}(\Omega)^{d\times d}$.
\end{lemma}

We further need the following result.

\begin{lemma}[{\citep[Lemma 2.3]{SWZ}}]\label{endpoint_estimate}
    Let $\Omega\subset\reals^d$ be a bounded Lipschitz domain and $q\in (d,\infty)$. Then there exists a constant $C_4>0$
    such that
    \begin{align}
        ||f||_{L^\infty} \leq C_4\big(1+||f||_{\BMO}\ln\!\big(e + ||\gradx f||_{L^q}\big)\big) \notag
    \end{align}
    for all $f\in W^{1,q}_0(\Omega)$.
\end{lemma}

\subsection{Proof of Theorem \texorpdfstring{\ref{blow_up_criterion}}{4.1}}\label{sec_proof}
Inspired by \citep[Remark 2.6]{Valli_Zajaczkowski} and the reasoning in \citep[Section 2]{Basaric_Feireisl_Mizerova}, we start the proof of Theorem \ref{blow_up_criterion} by inspecting the proof of Theorem \ref{existence_of_strong_solutions} again. We find that there exists a monotonously decreasing function $\underline{T}:[0,\infty)\to(0,\infty)$ such that
\begin{align}
    \Tmax \geq \underline{T}(\mathcal{D}_0), \qquad \text{where} \qquad \mathcal{D}_0 = \max\!\left\{||(\fatu_0,\varrho_0,Z_0)||_{H^3}, \left|\left|\frac{1}{\varrho_0}\right|\right|_{L^\infty}, \left|\left|\frac{1}{Z_0}\right|\right|_{L^\infty}\right\}. \notag
\end{align}
Therefore, a simple contradiction argument shows that
\begin{align}
    \Tmax<\infty \qquad \Rightarrow \qquad \limsup_{t\,\uparrow\,\Tmax}\!\left(||(\fatu,\varrho,Z)(t)||_{H^3} + \left|\left|\frac{1}{\varrho(t)}\right|\right|_{L^\infty} + \left|\left|\frac{1}{Z(t)}\right|\right|_{L^\infty}\right) = \infty. \label{blow_up_of_strong_norms}
\end{align}
The idea is to prove Theorem \ref{blow_up_criterion} by contradiction, too. Thus, from now on, we assume that
\begin{align}
    \Tmax<\infty \qquad \text{and} \qquad \sup_{t\,\in\,[0,\Tmax)}\big(||\fatu(t)||_{L^\infty} + ||\varrho(t)||_{L^\infty}\big) \leq K \label{assumption}
\end{align}
for some constant $K>0$. Our goal is to prove that \eqref{assumption} implies
\begin{align}
    \Tmax<\infty \qquad \text{and} \qquad \sup_{t\,\in\,[0,\Tmax)}\!\left(||(\fatu,\varrho,Z)(t)||_{H^3} + \left|\left|\frac{1}{\varrho(t)}\right|\right|_{L^\infty} + \left|\left|\frac{1}{Z(t)}\right|\right|_{L^\infty}\right) \leq K' \label{uniform_bound}
\end{align}
for some constant $K'>0$, which obviously contradicts \eqref{blow_up_of_strong_norms}.

\subsubsection{Boundedness of \texorpdfstring{$Z$}{Z}}
Our first observation is that, under assumption \eqref{assumption}, $Z$ is uniformly bounded. To see this, let again $m,M>0$ be such that $m\leq\varrho_0, Z_0\leq M$. Then $r=(2M/m)\varrho-Z$ satisfies
\begin{align}
    r_t + \divx(r\fatu) = 0 \qquad \text{with} \qquad r(0,\cdot) = \frac{2M}{m}\,\varrho_0 - Z_0 \geq M > 0. \notag
\end{align}
Thus, the solution formula obtained via the method of characteristics (cf. \eqref{solution_formula}) implies
\begin{align}
    r(t,\fatx) \geq M\exp\!\left(-\int_0^{\,t} ||\divx(\fatu)(s)||_{L^\infty}\;\ds\right) \geq 0 \notag 
\end{align}
for all $(t,\fatx)\in[0,\Tmax)\times\Omega$. Consequently, $Z\leq (2M/m)\varrho$ and \eqref{assumption} yields
\begin{align}
    \sup_{t\,\in\,[0,\Tmax)}\left(||Z(t)||_{L^\infty}\right) \leq C. \label{Z_bound}
\end{align}

\subsubsection{\texorpdfstring{$L^2$}{L2}-estimate for \texorpdfstring{$\gradx\fatu$}{grad(u)}}
We test \eqref{system_1} with $|\fatu|^2/2$ and \eqref{system_2} with $\fatu$. We add up the resulting equations and integrate in space. After integration by parts we obtain
\begin{align}
    \ddt\into \left(\frac{1}{2}\,\varrho|\fatu|^2\right)\dx + \into\left(\mu|\gradx\fatu|^2 + (\lambda+\mu)|\divx(\fatu)|^2\right)\dx = \into p(Z)\,\divx(\fatu)\;\dx \notag.
\end{align}
Using Young's inequality and \eqref{Z_bound}, the right-hand side of this equation can be estimated as follows:
\begin{align}
    \into p(Z)\,\divx(\fatu)\;\dx &\leq C\into |\divx(\fatu)|\;\dx \leq C\into\left(\frac{C}{2(\lambda+\mu)} + \frac{\lambda+\mu}{2C}|\divx(\fatu)|^2\right)\dx \notag \\[2mm]
    &\leq C + \frac{\lambda+\mu}{2}\into |\divx(\fatu)|^2\;\dx. \notag
\end{align}
Consequently,
\begin{align}
    ||\sqrt{\!\varrho}\fatu||_{L^\infty(0,T;L^2(\Omega)^d)} + ||\gradx\fatu||_{L^2((0,T)\times\Omega)^{d\times d}} \leq C \label{standard_estimate}
\end{align}
for all $T\in[0,\Tmax)$.

\subsubsection{Velocity splitting}
To improve the estimate on the velocity, we follow the approach in \citep{SWZ} and introduce the velocity splitting $\fatu=\fatv + \fatw$, where $\fatv(t)\equiv L^{-1}[\gradx p(Z(t))]$ is the solution to the Lamé system
\begin{align}
    L\fatv(t) &= \gradx p(Z(t)) \qquad \text{in $\Omega$}, \notag \\[2mm]
    \fatv(t)|_{\po} &= \bm{0}. \notag
\end{align}
It follows that $\fatw$ satisfies 
\begin{alignat}{2}
    \varrho\fatw_t - L\fatw &= \varrho\fatf && \qquad \text{in $Q_{\Tmax}$}, \label{w_equation1}\\[2mm]
    \fatw|_{[0,\Tmax)\times\po} &= \bm{0}, && \label{w_equation2} \\[2mm]
    \fatw(0,\cdot) &= \fatw_0 \equiv \fatu_0 - \fatv(0,\cdot) && \qquad \text{in $\Omega$}, \label{w_equation3}
\end{alignat}
where 
\begin{align}
    \fatf &= -(\fatu\cdot\gradx)\fatu - L^{-1}[\gradx(\pt p(Z))] \notag \\[2mm]
    &= -(\fatu\cdot\gradx)\fatu + L^{-1}[\gradx\divx(p(Z)\fatu)] + L^{-1}[\gradx((Zp'(Z)-p(Z))\divx(\fatu))]. \notag
\end{align}

\subsubsection{A priori estimates for \texorpdfstring{$\fatw$}{w}} 
Since $\fatv(t) = L^{-1}[\gradx p(Z(t))]$, it follows from Lemma \ref{lame_regularity} that for every $q\in(1,\infty)$ there exists a constant $C>0$ such that
\begin{align}
    ||\gradx\fatv(t)||_{L^q} &\leq C||p(Z(t))||_{L^q}, \label{v_regularity1} \\[2mm]
    ||\gradx^2\fatv(t)||_{L^q} &\leq C||\gradx p(Z(t))||_{L^q} \label{v_regularity2}
\end{align}
for all $t\in[0,\Tmax)$. Moreover, we have the following lemma concerning a priori estimates for the velocity component $\fatw$.

\begin{lemma} \label{w_regularity1}
    Under assumption \eqref{assumption} there exists $C>0$ such that for any $T\in[0,\Tmax)$
    \begin{align}
        ||\gradx\fatw||_{L^\infty(0,T;L^2(\Omega)^{d\times d})} + ||\sqrt{\!\varrho}\fatw_t||_{L^2((0,T)\times\Omega)^d} + ||\gradx^2\fatw||_{L^2((0,T)\times\Omega)^{d\times d\times d}} \leq C. \notag
    \end{align}
\end{lemma}

\begin{proof}
    Testing \eqref{w_equation1} with $\fatw_t$ and integrating over $\Omega$ yields
    \begin{align}
        \ddt \into \left(\mu|\gradx\fatw|^2 + (\lambda+\mu)|\divx(\fatw)|^2\right)\dx + \into\varrho|\fatw_t|^2\;\dx = \into \varrho\fatf\cdot\fatw_t\;\dx. \notag
    \end{align}
    Using Hölder's inequality and Young's inequality, we get
    \begin{align}
        \ddt \into \left(\mu|\gradx\fatw|^2 + (\lambda+\mu)|\divx(\fatw)|^2\right)\dx + \frac{1}{2}\into\varrho|\fatw_t|^2\;\dx \leq \frac{1}{2}\into \varrho|\fatf|^2\;\dx. \label{step1}
    \end{align}
    Recalling that
    \begin{align}
        \fatf &= -(\fatu\cdot\gradx)\fatu - L^{-1}[\gradx(\pt p(Z))] \notag \\[2mm]
        &= -(\fatu\cdot\gradx)\fatu + L^{-1}[\gradx\divx(p(Z)\fatu)] + L^{-1}[\gradx((Zp'(Z)-p(Z))\divx(\fatu))], \notag
    \end{align}
    we may estimate the right-hand side of \eqref{step1} as follows:
    \begin{align}
        ||\sqrt{\!\varrho}(\fatu\cdot\gradx)\fatu||_{L^2}^2 &\leq C||\gradx\fatu||_{L^2}^2, \notag \\[2mm]
        ||\sqrt{\!\varrho}L^{-1}[\gradx\divx(p(Z)\fatu)]||_{L^2}^2 &\leq C||p(Z)\fatu||_{L^2}^2 \leq C||\gradx\fatu||_{L^2}^2, \notag \\[2mm]
        ||\sqrt{\!\varrho}L^{-1}[\gradx((Zp'(Z)-p(Z))\divx(\fatu))]||_{L^2}^2 &\leq C||L^{-1}[\gradx((Zp'(Z)-p(Z))\divx(\fatu))]||_{L^2}^2 \notag \\[2mm]
        &\leq C||\gradx L^{-1}[\gradx((Zp'(Z)-p(Z))\divx(\fatu))]||_{L^2}^2 \leq C||\gradx\fatu||_{L^2}^2. \notag
    \end{align}
    Here, we have used \eqref{Z_bound} and Lemma \ref{lame_regularity}. For the second and third estimate, we additionally relied on Poincaré's inequality. Plugging the above estimates into \eqref{step1}, integrating in time and using \eqref{standard_estimate}, we see that
    \begin{align}
        ||\gradx\fatw||_{L^\infty(0,T;L^2(\Omega)^{d\times d})} + ||\sqrt{\!\varrho}\fatw_t||_{L^2((0,T)\times\Omega)^d} \leq C \notag
    \end{align}
    for all $T\in[0,\Tmax)$. Moreover, since $L\fatw = \varrho\fatw_t - \varrho\fatf$, it follows from elliptic regularity that
    \begin{align}
        ||\gradx^2\fatw||_{L^2((0,T)\times\Omega)^{d\times d\times d}} &\leq C\big(||\varrho\fatw_t||_{L^2((0,T)\times\Omega)^d} + ||\varrho\fatf||_{L^2((0,T)\times\Omega)^d}\big) \notag \\[2mm]
        &\leq C\big(||\sqrt{\!\varrho}\fatw_t||_{L^2((0,T)\times\Omega)^d} + ||\sqrt{\!\varrho}\fatf||_{L^2((0,T)\times\Omega)^d}\big) \leq C \notag
    \end{align}
    for all $T\in[0,\Tmax)$.
\end{proof}

Using \eqref{v_regularity1}, \eqref{v_regularity2} and the Sobolev embedding theorem, we deduce from Lemma \ref{w_regularity1} the following corollary.

\begin{corollary}\label{u_regularity_1}
    Under assumption \eqref{assumption} there exists $C>0$ such that for any $T\in[0,\Tmax)$
    \begin{align}
        ||\gradx\fatu||_{L^\infty(0,T;L^2(\Omega)^{d\times d})} 
        + ||\gradx\fatu||_{L^2(0,T;L^6(\Omega)^{d\times d})} \leq C. \notag
    \end{align}
\end{corollary}

\subsubsection{Higher-order a priori estimates for \texorpdfstring{$\fatw$}{w}}
Our next goal is to obtain higher-order a priori estimates for the velocity component $\fatw$.

\begin{lemma}\label{w_regularity2}
    Under assumption \eqref{assumption} there exists $C>0$ such that for any $T\in[0,\Tmax)$
    \begin{align}
        ||\gradx\fatw||_{L^2(0,T;L^\infty(\Omega)^{d\times d})} + ||\gradx^2\fatw||_{L^2(0,T;L^6(\Omega)^{d\times d\times d})} \leq C. \notag
    \end{align}
\end{lemma}

\begin{proof}
    Using the material derivative, the momentum equation \eqref{system_2} can be written as
    \begin{align}
        \varrho\fatudot + \gradx p(Z) - L\fatu = \bm{0}. \notag
    \end{align}
    Applying the material derivative to this equation, we obtain
    \begin{align}
        &\varrho\fatudot_t + \varrho(\fatu\cdot\gradx)\fatudot + \gradx p_t + \divx(\gradx p \otimes \fatu) \notag \\[2mm]
        &\qquad = \mu[\deltax\fatudot + \divx(\deltax\fatu\otimes\fatu)] + (\lambda+\mu)[\gradx\divx(\fatu_t) + \divx((\gradx\divx(\fatu))\otimes\fatu)]. \label{material_derviative_of_momentum_equation}
    \end{align}
    Next, we test \eqref{material_derviative_of_momentum_equation} with $\fatudot$ and integrate the result over $\Omega$. Integrating by parts some terms and using the continuity equation \eqref{system_1}, we get
    \begin{align}
        &\frac{1}{2} \ddt \into \varrho|\fatudot|^2\;\dx - (\lambda+\mu)\into\fatudot\cdot[\gradx\divx(\fatu_t) + \divx((\gradx\divx(\fatu))\otimes\fatu)]\;\dx \notag \\[2mm] 
        &\qquad = \mu\into \fatudot\cdot[\deltax\fatudot + \divx(\deltax\fatu\otimes\fatu)]\;\dx + \into \big[p_t\,\divx(\fatudot) + (\gradx p)^T\cdot\gradx\fatudot\cdot\fatu\big]\,\dx. \label{step11}
    \end{align}
    Integrating by parts and applying Young's inequality, we see that 
    \begin{align}
        &-\into \fatudot\cdot[\deltax\fatudot + \divx(\deltax\fatu\otimes\fatu)]\;\dx \notag \\[2mm]
        &\qquad = \into \big(\gradx\fatudot:\gradx\fatu_t + (\deltax\fatu\otimes\fatu):\gradx\fatudot\big)\,\dx \notag \\[2mm]
        &\qquad = \into \left(|\gradx\fatudot|^2 - \gradx\fatudot:\gradx(\gradx\fatu\cdot\fatu) + \deltax\fatu\cdot(\gradx\fatudot\cdot\fatu)\right)\dx \notag \\[2mm]
        &\qquad = \into \left(|\gradx\fatudot|^2 - \gradx\fatudot:\gradx(\gradx\fatu\cdot\fatu) - \gradx\fatu:\gradx(\gradx\fatudot\cdot\fatu)\right)\dx \notag \\[2mm]
        &\qquad = \into \left(|\gradx\fatudot|^2 - \gradx\fatudot:(\gradx\fatu\gradx\fatu) - \gradx\fatudot:[(\fatu\cdot\gradx)\gradx\fatu] - \gradx\fatu:(\gradx\fatudot\gradx\fatu)\right)\dx \notag \\[2mm]
        &\qquad \quad - \into\gradx\fatu:[(\fatu\cdot\gradx)\gradx\fatudot]\;\dx \notag \\[2mm]
        &\qquad = \into \left(|\gradx\fatudot|^2 - \gradx\fatudot:(\gradx\fatu\gradx\fatu) - \gradx\fatudot:[(\fatu\cdot\gradx)\gradx\fatu] - \gradx\fatu:(\gradx\fatudot\gradx\fatu)\right)\dx \notag \\[2mm]
        &\qquad \quad + \into\big(\gradx\fatudot:[(\fatu\cdot\gradx)\gradx\fatu] + \divx(\fatu)\gradx\fatudot:\gradx\fatu\big)\,\dx \notag \\[2mm]
        &\qquad = \into \left(|\gradx\fatudot|^2 - \gradx\fatudot:(\gradx\fatu\gradx\fatu) - \gradx\fatu:(\gradx\fatudot\gradx\fatu) + \divx(\fatu)\gradx\fatudot:\gradx\fatu\right)\dx \notag \\[2mm]
        & \qquad \geq \into \left(|\gradx\fatudot|^2 - (2+\sqrt{\!d})|\gradx\fatudot||\gradx\fatu|^2\right)\dx \notag \\[2mm]
        &\qquad \geq \into \left(\frac{3}{4}\,|\gradx\fatudot|^2 - (2+\sqrt{\!d})^2|\gradx\fatu|^4\right)\dx. \label{mu_term}
    \end{align}
    Using in addition the identities
    \begin{align}
        \divx(\gradx\divx(\fatu)\otimes\fatu) &= \gradx(\fatu\cdot\gradx\divx(\fatu)) - \divx(\divx(\fatu)(\gradx\fatu)^T) + \gradx(\divx(\fatu)^2), \notag \\[2mm]
        \divx(\fatudot) &= \divx(\fatu_t) + \fatu\cdot\gradx\divx(\fatu) + \gradx\fatu:(\gradx\fatu)^T, \notag
    \end{align}
    we deduce that
    \begin{align}
        &-\into \fatudot\cdot[\deltax\fatudot + \divx(\deltax\fatu\otimes\fatu)]\;\dx \notag \\[2mm]
        &\qquad = \into\left(|\divx(\fatudot)|^2 - \divx(\fatudot)\gradx\fatu:(\gradx\fatu)^T - \divx(\fatu)\gradx\fatudot:(\gradx\fatu)^T + \divx(\fatudot)\divx(\fatu)^2\right)\dx \notag \\[2mm]
        &\qquad \geq \into\left(|\divx(\fatudot)|^2 - (1+d)|\divx(\fatudot)||\gradx\fatu|^2 - \sqrt{\!d}\,|\gradx\fatudot||\gradx\fatu|^2\right)\dx \notag \\[2mm]
        &\qquad \geq \into\left(\frac{1}{2}\,|\divx(\fatudot)|^2 - \frac{\mu}{4(\lambda+\mu)}\,|\gradx\fatudot|^2 - \left(\frac{(1+d)^2}{2}+\frac{d(\lambda+\mu)}{\mu}\right)|\gradx\fatu|^4\right)\dx. \label{lambda_term}
    \end{align}
    Finally, using the $Z$-equation \eqref{system_3}, the bound on $Z$ \eqref{Z_bound} and Corollary \ref{u_regularity_1}, the pressure term can be handled as follows:
    \begin{align}
        &\into \big[p_t\,\divx(\fatudot) + (\gradx p)^T\cdot\gradx\fatudot\cdot\fatu\big]\,\dx = \into \big[p'Z_t\,\divx(\fatudot) + (\gradx p)^T\cdot\gradx\fatudot\cdot\fatu\big]\,\dx \notag \\[2mm]
        &\qquad = \into \big[p'Z\,\divx(\fatu)\,\divx(\fatudot) - (\fatu\cdot\gradx p)\,\divx(\fatudot) + (\gradx p)^T\cdot\gradx\fatudot\cdot\fatu\big]\,\dx \notag \\[2mm]
        &\qquad = \into \big[p'Z\,\divx(\fatu)\,\divx(\fatudot) + p\,\divx(\divx(\fatudot)\fatu) - p\,\divx(\gradx\fatudot\cdot\fatu)\big]\,\dx \notag \\[2mm]
        &\qquad = \into \big[p'Z\,\divx(\fatu)\,\divx(\fatudot) + p\,\divx(\fatudot)\,\divx(\fatu) - p\gradx\fatu:(\gradx\fatudot)^T\big]\,\dx \notag \\[2mm]
        &\qquad \leq C||\gradx\fatu||_{L^2}||\gradx\fatudot||_{L^2} \leq C||\gradx\fatudot||_{L^2}. \label{pressure_term}
    \end{align}
    Plugging \eqref{mu_term}--\eqref{pressure_term} into \eqref{step11}, we get 
    \begin{align}
        \ddt \into \varrho|\fatudot|^2\;\dx + \mu\into|\gradx\fatudot|^2\;\dx + (\lambda+\mu)\into|\divx(\fatudot)|^2\;\dx \leq C\big(||\gradx\fatu||_{L^4}^4 + ||\gradx\fatudot||_{L^2}\big). \label{step21}
    \end{align}
    Since $L\fatw = \varrho\fatudot$, $\fatw|_{[0,\Tmax)\times\po}$, elliptic regularity and \eqref{assumption} yield
    \begin{align}
        ||\gradx^2\fatw||_{L^2} \leq C||\varrho\fatudot||_{L^2} \leq C||\sqrt{\!\varrho}\fatudot||_{L^2}. \notag
    \end{align}
    Combining this information with Hölder's inequality, Corollary \ref{u_regularity_1}, \eqref{v_regularity1}, \eqref{assumption}, the Sobolev embedding theorem and Young's inequality, we deduce that
    \begin{align}
        ||\gradx\fatu||_{L^4}^4 &\leq C||\gradx\fatu||_{L^2}||\gradx\fatu||_{L^6}^3 \leq C||\gradx\fatu||_{L^6}^3 \leq C||\gradx\fatu||_{L^6}^2\big(||\gradx\fatw||_{L^6} + ||\gradx\fatv||_{L^6}\big) \notag \\[2mm]
        &\leq C||\gradx\fatu||_{L^6}^2\big(1+||\gradx^2\fatw||_{L^2}\big) \leq C||\gradx\fatu||_{L^6}^2\big(1+||\sqrt{\!\varrho}\fatudot||_{L^2}\big) \notag \\[2mm]
        &\leq C||\gradx\fatu||_{L^6}^2\big(1+||\sqrt{\!\varrho}\fatudot||_{L^2}^2\big). \label{step31}
    \end{align}
    Together with Young's inequality, \eqref{step21} and \eqref{step31} yield
    \begin{align}
        &\ddt \into \varrho|\fatudot|^2\;\dx + \frac{\mu}{2}\into|\gradx\fatudot|^2\;\dx + (\lambda+\mu)\into|\divx(\fatudot)|^2\;\dx \notag \\[2mm]
        &\qquad\leq C\big(1+||\gradx\fatu||_{L^6}^2 + ||\gradx\fatu||_{L^6}^2||\sqrt{\!\varrho}\fatudot||_{L^2}^2\big). \notag
    \end{align}
    Thus, it follows from Gronwall's inequality and Corollary \ref{u_regularity_1} that
    \begin{align}
        \into (\varrho|\fatudot|^2)(t) \;\dx + \inttinto|\gradx\fatudot|^2\;\dxds \leq C \label{u_regularity_2}
    \end{align}
    for all $t\in[0,\Tmax)$. Finally, using \eqref{assumption}, Lemma \ref{lame_regularity} with $q=6$ and the Sobolev embedding theorem, we see that
    \begin{align}
        ||\gradx^2\fatw||_{L^2(0,T;L^6(\Omega)^{d\times d\times d})} \leq C||\varrho\fatudot||_{L^2(0,T;L^6(\Omega)^d)} \leq C||\fatudot||_{L^2(0,T;L^6(\Omega)^d)} \leq C||\gradx\fatudot||_{L^2(0,T;L^2(\Omega)^{d\times d})} \leq C \notag
    \end{align}
    for all $T\in[0,\Tmax)$. Moreover, we get from the Sobolev embedding theorem that
    \begin{align}
        ||\gradx\fatw||_{L^2(0,T;L^\infty(\Omega)^{d\times d})} \leq C \notag
    \end{align}
    for all $T\in[0,\Tmax)$.
\end{proof}

\subsubsection{Higher-order a priori estimates for \texorpdfstring{$\fatu,\varrho,Z$}{u,rho,Z}}
We are now in the position to prove higher-order a priori estimates for $\fatu,\varrho$ and $Z$. We start with an a priori estimate for $\gradx\varrho$ and $\gradx Z$. 

\begin{lemma}\label{gradrho_gradz_regularity}
    Under assumption \eqref{assumption} there exists $C>0$ such that for any $T\in[0,\Tmax)$
    \begin{align}
        ||\gradx Z||_{L^\infty(0,T;L^6(\Omega)^d)} + ||\gradx \varrho||_{L^\infty(0,T;L^6(\Omega)^d)} \leq C. \notag
    \end{align}
\end{lemma}

\begin{proof}
    Taking the gradient of the $Z$-equation \eqref{system_3} we obtain
    \begin{align}
        \gradx Z_t = - (\fatu\cdot\gradx)\gradx Z - (\gradx\fatu)^T\cdot\gradx Z - \divx(\fatu)\gradx Z - Z\gradx\divx(\fatu) = \bm{0}. \label{gradz_t_equation}
    \end{align}
    We test this equation with $6|\gradx Z|^4\gradx Z$ and integrate over $\Omega$. After an integration-by-parts argument we arrive at
    \begin{align}
        \ddt \into |\gradx Z|^6\;\dx &= -5\into\divx(\fatu)|\gradx Z|^6\;\dx - 6\into Z|\gradx Z|^4\gradx Z\cdot\gradx\divx(\fatu)\;\dx \notag \\[2mm]
        &\quad - 6\into|\gradx Z|^4(\gradx Z)^T\cdot\gradx\fatu\cdot\gradx Z\;\dx. \label{step41}
    \end{align}
    To handle the right-hand side of \eqref{step41}, we need the following estimates that follow from \eqref{assumption} and the Lemmata \ref{lame_regularity}, \ref{gradu_regularity}, \ref{endpoint_estimate}:
    \begin{align}
        ||\gradx^2\fatv||_{L^6} &\leq C||\gradx Z||_{L^6}, \notag \\[2mm]
        ||\gradx\fatv||_{L^\infty} &\leq C(1+||\gradx\fatv||_{\BMO}\ln(e+||\gradx^2\fatv||_{L^6})) \leq C(1+||Z||_{L^\infty}\ln(e+||\gradx Z||_{L^6})) \notag \\[2mm]
        &\leq C(1+\ln(e+||\gradx Z||_{L^6})). \notag
    \end{align}
    Using these estimates and \eqref{assumption}, we deduce from \eqref{step41} that
    \begin{align}
        \ddt \into |\gradx Z|^6\;\dx &\leq C\left(\into |\gradx\fatu||\gradx Z|^6\;\dx + \into Z|\gradx Z|^5|\gradx\divx(\fatu)|\;\dx\right) \notag \\[2mm]
        &\leq C\left(||\gradx\fatu||_{L^\infty}||\gradx Z||_{L^6}^6 + ||\gradx^2\fatu||_{L^6}||\gradx Z||_{L^6}^5\right) \notag \\[2mm]
        &\leq C\left(\big[||\gradx\fatw||_{L^\infty} + ||\gradx\fatv||_{L^\infty}\big]||\gradx Z||_{L^6}^6 + \big[||\gradx^2\fatw||_{L^6} + ||\gradx^2\fatv||_{L^6}\big]||\gradx Z||_{L^6}^5\right) \notag \\[2mm]
        &\leq C\left(\big[1 + ||\gradx\fatw||_{L^\infty} + \ln(e+||\gradx Z||_{L^6})\big]||\gradx Z||_{L^6}^6 + ||\gradx^2\fatw||_{L^6}||\gradx Z||_{L^6}^5\right). \notag
    \end{align}
    Now, we would like to apply Gronwall's inequality. To this end, we have to modify the above inequality. Using Young's inequality we observe that
    \begin{align}
       ||\gradx Z||_{L^6}^5 &\leq \tfrac{1}{6}\big(1 + 5||\gradx Z||_{L^6}^6\big)\leq e + ||\gradx Z||_{L^6}^6, \notag \\[2mm]
       e+||\gradx Z||_{L^6} &\leq e + \tfrac{1}{6}\big(5+||\gradx Z||_{L^6}^6\big)\leq 2\big(e + ||\gradx Z||_{L^6}^6\big). \notag
    \end{align}
    Consequently, 
    \begin{align}
        &\ddt \left(e + \into |\gradx Z|^6\;\dx\right) \leq C\big(1 + ||\gradx\fatw||_{L^\infty} + \ln(e+||\gradx Z||_{L^6}^6) + ||\gradx^2\fatw||_{L^6} \big)\big(e+||\gradx Z||_{L^6}^6\big) \notag
    \end{align}
    and dividing by $e+||\gradx Z||_{L^6}^6$ yields
    \begin{align}
        \ddt \ln\!\left(e + \into |\gradx Z|^6\;\dx\right) \leq C\big(1 + ||\gradx\fatw||_{L^\infty} + \ln(e+||\gradx Z||_{L^6}^6) + ||\gradx^2\fatw||_{L^6} \big). \notag
    \end{align}
    Thus, Lemma \ref{w_regularity2} and Gronwall's inequality imply that
    \begin{align}
        \ln\!\big(e+||\gradx Z(t)||_{L^6}^6\big) \leq C \notag
    \end{align}
    for all $t\in[0,\Tmax)$ and, in particular,
    \begin{align}
        ||\gradx Z(t)||_{L^6}^6 \leq C \notag
    \end{align}
    for all $t\in[0,\Tmax)$. 
    Finally, the same computations can be done for the density $\varrho$ using the continuity equation \eqref{system_1}. Note that this time we already know that
    \begin{align}
        ||\gradx\fatv(t)||_{L^\infty} + ||\gradx^2\fatv(t)||_{L^6} \leq C \notag
    \end{align}
    for all $t\in[0,\Tmax)$. 
    Thus, Gronwall's inequality is directly applicable to the quantity $||\gradx \varrho(t)||_{L^6}^6$ yielding 
    \begin{align}
        ||\gradx \varrho(t)||_{L^6}^6 \leq C \notag
    \end{align}
    for all $t\in[0,\Tmax)$. 
\end{proof}

With the help of the previous lemma, we can obtain second-order estimates for $\fatu$ and complete the first-order estimates for $\varrho$ and $Z$.

\begin{corollary}
    Under assumption \eqref{assumption} there exists $C>0$ such that for any $T\in[0,\Tmax)$
    \begin{align}
        ||\fatu||_{L^2(0,T;W^{2,6}(\Omega)^d)} + ||\fatu||_{L^\infty(0,T;H^2(\Omega)^d)} &\leq C, \label{u_estimates} \\[2mm]
        ||\varrho_t||_{L^\infty(0,T;L^6(\Omega))} + ||Z_t||_{L^\infty(0,T;L^6(\Omega))} &\leq C,  \label{rhot_Zt_estimates} \\[2mm] 
        \left|\left|\frac{1}{\varrho}\right|\right|_{L^\infty(0,T;L^\infty(\Omega))} + \left|\left|\frac{1}{Z}\right|\right|_{L^\infty(0,T;L^\infty(\Omega))} &\leq C. \label{rho_Z_lower_bounds}
    \end{align}
\end{corollary}

\begin{proof}
    Using Poincaré's inequality, \eqref{v_regularity1}, \eqref{v_regularity2}, Lemma \ref{w_regularity2}, \eqref{Z_bound} and Lemma \ref{gradrho_gradz_regularity} we see that
    \begin{align}
        ||\fatu||_{L^2(0,T;W^{2,6}(\Omega)^d)} &\leq ||\fatw||_{L^2(0,T;W^{2,6}(\Omega)^d)} + ||\fatv||_{L^2(0,T;W^{2,6}(\Omega)^d)} \notag \\[2mm]
        &\leq C\left(||\gradx\fatw||_{L^2(0,T;L^6(\Omega)^{d\times d})} + ||\gradx^2\fatw||_{L^2(0,T;L^6(\Omega)^{d\times d\times d})}\right) \notag \\[2mm]
        &\quad + C\left(||\gradx\fatv||_{L^2(0,T;L^6(\Omega)^{d\times d})} + ||\gradx^2\fatv||_{L^2(0,T;L^6(\Omega)^{d\times d\times d})}\right) \notag \\[2mm]
        &\leq C + C\left(||p(Z)||_{L^2(0,T;L^6(\Omega))} + ||\gradx p(Z)||_{L^2(0,T;L^6(\Omega)^d)}\right) \notag \\[2mm]
        &\leq C + C||\gradx Z||_{L^2(0,T;L^6(\Omega)^d)} \leq C \notag
    \end{align}
    for all $T\in[0,\Tmax)$. Next, since $\varrho\fatudot + \gradx p(Z) - L\fatu = \bm{0}$, $\fatu|_{[0,\Tmax)\times\po}=\bm{0}$, we may use elliptic regularity, \eqref{assumption}, \eqref{Z_bound}, \eqref{u_regularity_2} and Lemma \ref{gradrho_gradz_regularity} to deduce that
    \begin{align}
        ||\fatu||_{L^\infty(0,T;H^2(\Omega)^d)} &\leq C\left(||\varrho\fatudot||_{L^\infty(0,T;L^2(\Omega)^d)} + ||\gradx p(Z)||_{L^\infty(0,T;L^2(\Omega)^d)}\right) \notag \\[2mm]
        &\leq C\left(||\sqrt{\!\varrho}\fatudot||_{L^\infty(0,T;L^2(\Omega)^d)} + ||\gradx Z||_{L^\infty(0,T;L^2(\Omega)^d)}\right) \leq C \notag
    \end{align}
    for all $T\in[0,\Tmax)$. Thus, \eqref{u_estimates} is proven. Using the $Z$-equation \eqref{system_3}, 
    \eqref{assumption},
    \eqref{Z_bound}, Lemma \ref{gradrho_gradz_regularity}, the Sobolev embedding theorem and \eqref{u_estimates}, we see that
    \begin{align}
        ||Z_t||_{L^\infty(0,T;L^6(\Omega))} &\leq ||\fatu\cdot \gradx Z||_{L^\infty(0,T;L^6(\Omega))} + ||\divx(\fatu)Z||_{L^\infty(0,T;L^6(\Omega))} \notag \\[2mm]
        &\leq C\left(||\gradx Z||_{L^\infty(0,T;L^6(\Omega)^d)} + ||\fatu||_{L^\infty(0,T;H^2(\Omega)^d)}\right) \leq C \notag
    \end{align}
    for all $T\in[0,\Tmax)$. Analogously, we obtain
    \begin{align}
        ||\varrho_t||_{L^\infty(0,T;L^6(\Omega))} \leq C \notag
    \end{align}
    for all $T\in[0,\Tmax)$. It remains to prove \eqref{rho_Z_lower_bounds}. From \eqref{u_estimates} and the Sobolev embedding theorem it follows that 
    \begin{align}
        ||\divx(\fatu)||_{L^1(0,T;L^\infty(\Omega))} \leq C \notag
    \end{align}
    for all $T\in[0,\Tmax)$. Thus, it follows from the solution formula obtained via the method of characteristics (cf. \eqref{solution_formula}) that
    \begin{align}
        \varrho(t,\fatx) \geq \underline{\varrho}\exp\!\left(-\int_0^{\,t}||\divx(\fatu)(s)||_{L^\infty}\;\ds\right) \geq \underline{\varrho}\exp(-C\Tmax) > 0 \notag
    \end{align}
    for all $(t,\fatx)\in[0,\Tmax)\times\Omega$. Analogously, we see that 
    \begin{align}
        Z(t,\fatx) \geq \underline{Z}\exp(-C\Tmax) > 0 \notag
    \end{align}
    for all $(t,\fatx)\in[0,\Tmax)\times\Omega$. This completes the proof.
\end{proof}

Next, we derive second-order estimates for $\varrho,Z$ and obtain a first estimate for third-order derivatives of $\fatu$.

\begin{lemma} \label{LinfH2}
    Under assumption \eqref{assumption} there exists $C>0$ such that for any $T\in[0,\Tmax)$
    \begin{align}
        ||Z||_{L^\infty(0,T;H^2(\Omega))} + ||Z_t||_{L^\infty(0,T;H^1(\Omega))}
        \leq C, \label{Z_LinfH2} \\[2mm]
        ||\varrho||_{L^\infty(0,T;H^2(\Omega))} + ||\varrho_t||_{L^\infty(0,T;H^1(\Omega))} 
        \leq C, \label{rho_LinfH2} \\[2mm]
        ||\fatu||_{L^2(0,T;H^3(\Omega)^d)} \leq C. \label{u_L2H3}
    \end{align}
\end{lemma}

\begin{proof}
    Applying $\ppj\ppi$, $i,j\in\{1,\dots,d\}$, to the $Z$-equation \eqref{system_3}, we obtain
    \begin{align}
        &\ppj\ppi Z_t + \ppj\ppi\fatu\cdot\gradx Z + \fatu\cdot\gradx\ppj\ppi Z + \ppj\fatu\cdot\gradx\ppi Z + \ppi\fatu\cdot\gradx\ppj Z \notag \\[2mm]
        &\qquad + \ppj\ppi Z\,\divx(\fatu) + Z\,\ppj\ppi\divx(\fatu) + \ppi Z\, \ppj\divx(\fatu) + \ppj Z \,\ppi\divx(\fatu) = 0. \notag
    \end{align}
    Next, we test this equation with $2\ppj\ppi Z$ and integrate over $\Omega$. After integration by parts, we get
    \begin{align}
        \ddt \into |\ppj\ppi Z|^2\;\dx &= - \into |\ppj\ppi Z|^2\divx(\fatu)\;\dx - 2\into \ppj\ppi Z\,\ppi\fatu\cdot\gradx\ppj Z\;\dx \notag \\[2mm]
        &\phantom{=\;} - 2\into \ppj\ppi Z \,\ppj\fatu\cdot\gradx\ppi Z\;\dx - 2\into Z\,\ppj\ppi Z\,\ppj\ppi\divx(\fatu)\;\dx \notag \\[2mm]
        &\phantom{=\;} - 2\into \ppi Z \,\ppj\ppi Z\,\ppj\divx(\fatu)\;\dx - 2\into \ppj Z \,\ppj\ppi Z\,\ppi\divx(\fatu)\;\dx \notag \\[2mm]
        &\phantom{=\;} - 2\into \ppj\ppi Z \,\ppj\ppi\fatu\cdot\gradx Z\;\dx = \sum_{k\,=\,1}^{7} I_k, \label{step51}
    \end{align}
    where, by Hölder's inequality and Sobolev's embedding theorem,
    \begin{align}
        |I_1| + |I_2| + |I_3| \leq C||\gradx\fatu||_{L^\infty}\into|\gradx^2 Z|^2\;\dx \leq C||\fatu||_{W^{2,6}}\into|\gradx^2 Z|^2\;\dx. \label{step52}
    \end{align}
    To estimate the remaining terms on the right-hand side of \eqref{step51}, we observe that, due to elliptic regularity, \eqref{rho_Z_lower_bounds}\footnote{To control $p''(Z)=\gamma(\gamma-1)Z^{\gamma-2}$, we need both the upper and the lower bound for $Z$ since we only know that $\gamma>1$.}, Hölder's inequality, \eqref{u_regularity_2}, \eqref{assumption}, \eqref{Z_bound}, Sobolev's inequality and Lemma \ref{gradrho_gradz_regularity},
    \begin{align}
        ||\fatu||_{H^3} &\leq C\left(||\varrho\fatudot||_{H^1} + ||\gradx p(Z)||_{H^1}\right) \notag \\[2mm]
        &\leq C\left(||\varrho\fatudot||_{L^2} + ||\varrho\gradx\fatudot||_{L^2} + ||\fatudot\otimes\gradx\varrho||_{L^2}\right) \notag \\[2mm] 
        &\quad + C\left(||p'(Z)\gradx Z||_{L^2} + ||p'(Z)\gradx^2 Z||_{L^2} + ||p''(Z)\gradx Z\otimes \gradx Z||_{L^2}\right) \notag \\[2mm]
        &\leq C\left(1 + ||\gradx\fatudot||_{L^2} + ||\fatudot||_{L^3}||\gradx\varrho||_{L^6} + ||\gradx Z||_{L^2} + ||\gradx^2 Z||_{L^2} + ||\gradx Z||_{L^4}^2\right) \notag \\[2mm]
        &\leq C\left(1 + ||\gradx\fatudot||_{L^2} + ||\gradx^2 Z||_{L^2}\right). \label{u_H3}
    \end{align}
    Thus, using Hölder's inequality, \eqref{Z_bound} and Young's inequality, we see that 
    \begin{align}
        |I_4| &\leq C||\gradx^2 Z||_{L^2}||\fatu||_{H^3} \leq C||\gradx^2 Z||_{L^2}\left(1 + ||\gradx\fatudot||_{L^2} + ||\gradx^2 Z||_{L^2}\right) \notag \\[2mm]
        &\leq C\left(1 + ||\gradx\fatudot||_{L^2}^2 + ||\gradx^2 Z||_{L^2}^2\right). \label{step53}
    \end{align}
    Similarly, using Hölder's inequality, Lemma \ref{gradrho_gradz_regularity}, the Sobolev embedding theorem and Young's inequality, we deduce that
    \begin{align}
        |I_5| + |I_6| + |I_7| &\leq C||\gradx Z||_{L^6}||\gradx^2 Z||_{L^2}||\gradx^2\fatu||_{L^3} \leq C||\gradx^2 Z||_{L^2}||\fatu||_{H^3} \notag \\[2mm]
        &\leq C\left(1 + ||\gradx\fatudot||_{L^2}^2 + ||\gradx^2 Z||_{L^2}^2\right). \label{step54}
    \end{align}
    Plugging \eqref{step52}, \eqref{step53} and \eqref{step54} into \eqref{step51}, we obtain
    \begin{align}
        \ddt \into|\gradx^2 Z|^2\;\dx \leq C(1 + ||\gradx\fatudot||_{L^2}^2) + C(1+||\fatu||_{W^{2,6}})\into|\gradx^2 Z|^2\;\dx. \notag
    \end{align}
    Now, due to \eqref{u_regularity_2} and \eqref{u_estimates}, an application of Gronwall's inequality shows that
    \begin{align}
        ||\gradx^2 Z||_{L^\infty(0,T;L^2(\Omega))} \leq C \notag
    \end{align}
    for all $T\in[0,\Tmax)$. Recalling \eqref{Z_bound} and Lemma \ref{gradrho_gradz_regularity}, we see that
    \begin{align}
        ||Z||_{L^\infty(0,T;H^2(\Omega))} \leq C \notag
    \end{align}
    for all $T\in[0,\Tmax)$. Using this estimate, \eqref{u_estimates}, Hölder's inequality and the Sobolev embedding theorem, it follows from \eqref{gradz_t_equation} that
    \begin{align}
        ||\gradx Z_t||_{L^\infty(0,T;L^2(\Omega)^d)} &\leq C\left(||\gradx^2 Z||_{L^\infty(0,T;L^2(\Omega)^{d\times d})}||\fatu||_{L^\infty(0,T;L^\infty(\Omega)^d)}\right. \notag \\[2mm]
        &\qquad \left.+\;||\gradx\fatu||_{L^\infty(0,T;L^4(\Omega)^{d\times d})}||\gradx Z||_{L^\infty(0,T;L^4(\Omega)^d)} \right.\notag \\[2mm]
        &\qquad \left.+\;||Z||_{L^\infty(0,T;L^\infty(\Omega))}||\gradx^2\fatu||_{L^\infty(0,T;L^2(\Omega)^{d\times d\times d})}\right) \notag \\[2mm]
        &\leq C||Z||_{L^\infty(0,T;H^2(\Omega))}||\fatu||_{L^\infty(0,T;H^2(\Omega)^d)} \leq C \notag
    \end{align}
    for all $T\in[0,\Tmax)$. Together with \eqref{rhot_Zt_estimates}, this yields
    \begin{align}
        ||Z_t||_{L^\infty(0,T;H^1(\Omega))} &\leq C \notag
    \end{align}
    for all $T\in[0,\Tmax)$. Thus, \eqref{Z_LinfH2} is proven. Estimate \eqref{rho_LinfH2} can be proven analogously. The proof makes use of estimate \eqref{u_L2H3} which is an immediate consequence of \eqref{u_H3}, \eqref{Z_LinfH2} and \eqref{u_regularity_2}.
\end{proof}

We are now able to obtain the desired estimate for $\fatu$.

\begin{corollary}\label{cor_u_LinfH3}
    Under assumption \eqref{assumption} there exists $C>0$ such that for any $T\in[0,\Tmax)$
    \begin{align}
        ||\fatudot||_{L^\infty(0,T;H^1(\Omega)^d)} \leq C, \label{udot_LinfH1} \\[2mm]
        ||\fatu||_{L^\infty(0,T;H^3(\Omega)^d)} \leq C. \label{u_LinfH3}
    \end{align}
\end{corollary}

\begin{proof}
    Differentiating the momentum equation \eqref{system_2} with respect to time, we see that $\fatudot$ is a solution to the problem
    \begin{align}
        \varrho\fatudot_t - L\fatudot + \varrho_t\fatudot &= \fatG, \notag \\[2mm]
        \fatudot|_{[0,\Tmax)\times\po} &= \bm{0}, \notag \\[2mm]
        \fatudot(0,\cdot) &= \fatudot_0, \notag
    \end{align}
    where
    \begin{align}
        \fatG = - L(\gradx\fatu\cdot\fatu) - \pt\gradx p(Z), \qquad \fatudot_0 = \frac{L\fatu_0 - \gradx p(Z_0)}{\varrho_0}. \notag
    \end{align}
    Thus, estimate \eqref{udot_LinfH1} follows from Lemma \ref{lemma_improved_regularity}, \eqref{rho_LinfH2}, \eqref{Z_LinfH2}, \eqref{u_L2H3} and \eqref{u_estimates}. Then, combining \eqref{udot_LinfH1} with \eqref{u_H3} and \eqref{Z_LinfH2}, we get \eqref{u_LinfH3}. 
\end{proof}

Finally, we obtain the desired estimates for $\varrho$ and $Z$.

\begin{lemma}\label{LinfH3}
    Under assumption \eqref{assumption} there exists $C>0$ such that for any $T\in[0,\Tmax)$
    \begin{align}
        ||Z||_{L^\infty(0,T;H^3(\Omega))} \leq C, \label{Z_LinfH3} \\[2mm]
        ||\varrho||_{L^\infty(0,T;H^3(\Omega))} \leq C. \label{rho_LinfH3}
    \end{align}
\end{lemma}

\begin{proof}
    The proof of this lemma is analogous to that of Lemma \ref{LinfH2}. One more spatial derivative is required and thus, on the rigorous level, one has to work with regularizations again. 
\end{proof}

Together, Corollary \ref{cor_u_LinfH3} and Lemma \ref{LinfH3} show that \eqref{assumption} implies \eqref{uniform_bound} which contradicts \eqref{blow_up_of_strong_norms}. Thus, Theorem \ref{blow_up_criterion} is proven.

\section{Generalizations and further results}\label{sec_conclusion}
We have seen in Section \ref{sec_local_in_time_strong_solutions} that -- under suitable assumptions on the initial data, the body force $\fatb$, etc. -- there does exist a unique local-in-time strong solution $(\fatu,\varrho,Z)$ to \eqref{system_1}--\eqref{pressure_with_Z}, \eqref{viscosity_coefficients}. Further, as demonstrated in Section \ref{sec_conditional_regularity}, this strong solution is global in time provided $\fatb=\bm{0}$ and its density and velocity components remain uniformly bounded in time.

\subsection{Generalizations}
We would like to point out that Theorem \ref{blow_up_criterion} can also be proven without the restriction $\fatb=\bm{0}$. That is, we have the following result, the proof of which is left to the reader.

\begin{theorem}\label{blow_up_criterion_2}
    Let $\fatb\in L^\infty(\reals^+;H^2(\Omega)^d)$, $\fatb_t\in L^\infty(\reals^+;H^1(\Omega)^d)$ and let the assumptions of Theorem \ref{existence_of_strong_solutions} be satisfied. Let $(\fatu,\varrho,Z)$ be the strong solution to \eqref{system_1}--\eqref{pressure_with_Z}, \eqref{viscosity_coefficients} with maximal existence time $\Tmax$. If $\Tmax<\infty$, then
    \begin{align}
        \limsup_{t\,\uparrow\,\Tmax}\,\left(||\varrho(t)||_{L^\infty} + ||\fatu(t)||_{L^\infty}\right) = \infty. \notag
    \end{align}
\end{theorem}

Since the exclusion of initial vacuum ensures the nonappearance of vacuum at later times, the abovementioned results can be transferred to the setting of \eqref{original_1}--\eqref{viscosity_coefficients} via the transformation 
\begin{align}
    (\varrho,Z) \mapsto (\varrho,\theta)\equiv\left(\varrho, \frac{Z}{\varrho}\right). \notag
\end{align}
That is, we have the following results.

\begin{theorem}\label{thm_original_existence}
    Let $\po\in C^4$, $0<m<M$, 
    $\fatb\in L^2_{\mathrm{loc}}(\reals^+;H^2(\Omega)^d)$ such that $\fatb_t\in L^2_{\mathrm{loc}}(\reals^+;L^2(\Omega)^d)$, $\fatu_0\in H^3(\Omega)^d\cap H^1_0(\Omega)^d$, $\varrho_0, \theta_0\in H^3(\Omega)$, $m \leq \varrho_0(\fatx),\theta_0(\fatx) \leq M$ for all $\fatx\in\Omega$ and $\gradx p(\varrho_0\theta_0)|_{\po} = (\varrho_0\fatb(0) + L\fatu_0)|_{\po}$. Then there exist a time $T^\star\in(0,T]$ and a unique strong solution $(\fatu,\varrho,\theta)$ to \eqref{original_1}--\eqref{viscosity_coefficients} in $Q_{T^\star}$ in the class
    \begin{gather}
        \varrho,\theta\in C([0,T^\star];H^3(\Omega)), \qquad \varrho,\theta>0 \quad \text{in $Q_{T^\star}$}, \qquad \varrho_t,\theta_t\in C([0,T^\star];H^2(\Omega)), \notag \\[2mm]
        \fatu\in L^2(0,T^\star;H^4(\Omega)^d)\cap C([0,T^\star];H^3(\Omega)^d), \qquad \fatu_t\in L^2(0,T^\star;H^2(\Omega)^d), \notag \\[2mm]
        \fatu_{tt}\in L^2(0,T^\star;L^2(\Omega)^d). \notag 
    \end{gather}
\end{theorem}

\begin{theorem}\label{thm_original_conditional_regularity}
    Let the assumptions of Theorem \ref{thm_original_existence} be satisfied and let $(\fatu,\varrho,\theta)$ be the strong solution to \eqref{original_1}--\eqref{viscosity_coefficients} with maximal existence time $\Tmax$. If $\Tmax<\infty$, then
    \begin{align}
        \limsup_{t\,\uparrow\,\Tmax}\,\left(||\varrho(t)||_{L^\infty} + ||\fatu(t)||_{L^\infty}\right) = \infty. \notag
    \end{align}
\end{theorem}

\begin{remark}
    Looking at the proofs of Theorems \ref{existence_of_strong_solutions} and \ref{blow_up_criterion}, it turns out that they remain valid as long as the argument of the pressure function $p$ is $Z=\varrho\theta$ and $p$ is any smooth function, $p\in C^1([0,\infty))\cap C^3(\reals^+)$, not necessarily monotone.
\end{remark}

\begin{remark}
    Results analogous to Theorems \ref{thm_original_existence} and \ref{thm_original_conditional_regularity} can be proven if the term $-\kappa\deltax\theta$ is added to the left-hand side of \eqref{original_4} and the additional boundary condition $\gradx\theta\cdot\fatn|_{[0,T]\times\po} = 0$ is introduced. 
\end{remark}

\subsection{Further results}
Combining the above theorems with the DMV-strong uniqueness principle \citep[Theorem 3.3]{LS_DMV_strong_uniqueness}\footnote{By a density argument, the class of test functions in the definition of DMV solutions can be enlarged in such a way that the proof of the DMV-strong uniqueness principle is still valid if the velocity component $\fatu$ of the strong solution is in the class of functions given in Theorem \ref{sec_local_in_time_strong_solutions}.}, we obtain the following corollary.

\begin{corollary}
    Let the assumptions of Theorem \ref{thm_original_existence} be satisfied with $\fatb=\bm{0}$     and $p(\varrho\theta)=a(\varrho\theta)^\gamma$, where $a>0,\gamma>1$. 
    Let further $\mathcal{V}$ be a dissipative measure-valued solution to \eqref{original_1}--\eqref{viscosity_coefficients} in the sense of \citep[Definition 2.1]{LS_existence} with energy dissipation defect $\mathfrak{E}$, dissipation defect $\mathfrak{D}$ and Reynolds concentration defect $\bm{\mathfrak{R}}$ and suppose there exist constants $r,u>0$ such that 
    \begin{align}
        \mathcal{V}_{(t,\fatx)}(\{\tilde{\varrho}\leq r\}\cap\{|\tilde{\fatu}|\leq u\}) = 1. \notag
    \end{align}
    Then the strong solution $(\fatu,\varrho,\theta)$ to \eqref{original_1}--\eqref{viscosity_coefficients} exists and $\mathcal{V}_{(t,\fatx)}=\delta_{(\fatu,\varrho,\theta)(t,\fatx)}$ for a.e. $(t,\fatx)\in(0,T)\times\Omega$, $\mathfrak{E}=0$, $\mathfrak{D}([0,T)\times\overline{\Omega})=0$ and $\bm{\mathfrak{R}}=\bm{0}$. 
\end{corollary}

Moreover, combining Theorems \ref{thm_original_existence} and \ref{thm_original_conditional_regularity} with \citep[Theorem 6.1]{LS_existence}, we obtain the following result. 

\begin{corollary}\label{cor_num_scheme}
    Let the assumptions of Theorem \ref{thm_original_existence} be satisfied with $\fatb=\bm{0}$
    and $p(\varrho\theta)=a(\varrho\theta)^\gamma$, where $a>0,\gamma>1$. Suppose further that the numerical approximations $(\fatu_{h},\varrho_{h},\theta_{h})_{h\,\downarrow\,0}$ ($h>0$ is a discretization parameter) obtained from the numerical scheme presented in \citep[Section 3]{LS_existence} satisfy
    \begin{align}
        \sup_{h\,>\,0}\big\{||(\fatu_h,\varrho_h)||_{L^\infty((0,T)\times\Omega_h)^{d+1}}\big\} < \infty. \notag
    \end{align}
    Then the strong solution $(\fatu,\varrho,\theta)$ to \eqref{original_1}--\eqref{viscosity_coefficients} exists and $(\fatu_h,\varrho_h,\theta_h)\to(\fatu,\varrho,\theta)$ strongly in $L^q(Q_T)$ for any $q\in[1,\infty)$ as $h\downarrow 0$.
\end{corollary}

\begin{remark}
    Note that the same result as in Corollary \ref{cor_num_scheme} can be obtained for any consistent numerical scheme for \eqref{original_1}--\eqref{viscosity_coefficients}. By a consistent numerical scheme we mean a scheme the solutions of which satisfy a weak formulation of \eqref{original_1}--\eqref{viscosity_coefficients} up to some consistency errors that vanish as $h\to 0$; see \citep[Definition 5.9]{Feireisl_Lukacova_book}.
\end{remark}

%
%
%

\bibliographystyle{plain}
\bibliography{References}

\appendix
\section{Appendix}\label{sec_appendix}
We recall some classical results that are used in Section \ref{sec_local_in_time_strong_solutions}. 

\begin{theorem}[method of continuity, {\citep[Theorem 5.2]{Gilbarg_Trudinger}}]\label{method_of_continuity}
    Let $B$ be a Banach space, $V$ a normed vector space, and $(T_\alpha)_{\alpha\in[0,1]}$ a norm continuous family of bounded linear operators from $B$ into $V$. Assume that there exists a positive constant $C$ such that for every $\alpha\in [0,1]$ and every $x\in B$
    \begin{align}
        ||x||_{B}\leq C||T_{\alpha}(x)||_{V}. \notag
    \end{align}
    Then $T_0$ is surjective if and only if $T_1$ is surjective as well. 
\end{theorem}

\begin{theorem}[Schauder's fixed-point theorem, {\citep[Satz II]{Schauder}}]\label{schauder}
    Let $R$ be a nonempty convex closed subset of a Banach space $\mathbb{X}$ and $\Phi:R\to R$ a continuous map such that $\Phi(R)\subset R$ is compact. Then $\Phi$ has a fixed point. 
\end{theorem}

\begin{theorem}[Aubin-Lions lemma, {\citep[Theorem II.5.16]{Boyer}}]\label{Aubin-Lions_lemma}
    Let $B_0\subset B_1\subset B_2$ be three Banach spaces. We assume that the embedding of $B_1$ in $B_2$ is continuous and that the embedding of $B_0$ in $B_1$ is compact. For $p,r\in[1,\infty]$ and $T>0$ we define
    \[E_{p,r} = 
    \big\{v\in L^p(0,T;B_0)\,\big|\,v_t\in L^r(0,T;B_2)\big\}
    .\]
    Then:
    \begin{itemize}
        \item[$\mathrm{(i)}$]{If $p<\infty$, then the embedding of $E_{p,r}$ in $L^p(0,T;B_1)$ is compact.}
        \item[$\mathrm{(ii)}$]{If $p=\infty$ and $r>1$, then the embedding of $E_{p,r}$ in $C([0,T];B_1)$ is compact.}
    \end{itemize}
\end{theorem}

For the proof of Theorem \ref{existence_of_strong_solutions} we further need the following results that were proven in \citep{Lions_Magenes} in a more general setting. To obtain the results as stated below, we take $X=H^2(\Omega)^d$, $Y=L^2(\Omega)^d$, $j=0$, $m=1$ and observe that 
\begin{align}
    [H^2(\Omega)^d,L^2(\Omega)^d]_{1/2} = H^1(\Omega)^d. \notag
\end{align}
For a short introduction to the interpolation spaces $[X,Y]_{\theta}$ with $X,Y$ Hilbert spaces and $\theta\in[0,1]$ real numbers see \citep[Section 2]{Lions_Magenes}. A more extensive introduction to interpolation theory can be found in \citep[Chapter 1]{Triebel}. The aforementioned observation was proven in \citep[Chapter 1, Theorem 9.6]{Lions_Magenes} for bounded domains $\Omega$ with smooth boundary and in \citep[Section 4.3.2, Theorem 2]{Triebel} for bounded domains satisfying the cone condition. For the connection between the presentations in \citep{Lions_Magenes} and \citep{Triebel}, we refer to \citep[Section 1.18.10]{Triebel}.

\begin{theorem}[{\citep[Chapter 1, Theorem 3.1]{Lions_Magenes}}]\label{Lions_embedding}
    There is a continuous embedding
    \begin{align}
        \big\{\fatw\in L^2(\reals^+;H^2(\Omega)^d)\,\big|\,\fatw_t\in L^2(\reals^+;L^2(\Omega)^d)\big\}\hookrightarrow C_b(\reals^+;H^1(\Omega)^d). \notag
    \end{align}
    In particular, there exists a constant $C_5>0$ such that
    \begin{align}
        ||\fatw||_{L^\infty(\reals^+;H^1(\Omega)^d)}^2 \leq C_5\left(||\fatw||_{L^2(\reals^+;H^2(\Omega)^d)}^2 + ||\fatw_t||_{L^2(\reals^+;L^2(\Omega)^d)}^2\right) \notag
    \end{align}
    for all $\fatw\in\big\{\fatw\in L^2(\reals^+;H^2(\Omega)^d)\,\big|\,\fatw_t\in L^2(\reals^+;L^2(\Omega)^d)\big\}$.
\end{theorem}

\begin{theorem}[{\citep[Chapter 1, Remark 3.3]{Lions_Magenes}}]\label{Lions_right_inverse}
    There exists a continuous linear right inverse $\mathscr{R}$ of the map
    \begin{align}
        \big\{\fatw\in L^2(\reals^+;H^2(\Omega)^d)\,\big|\,\fatw_t\in L^2(\reals^+;L^2(\Omega)^d)\big\}
        \to H^1(\Omega)^d, \qquad \fatw \mapsto \fatw(0). \notag
    \end{align}
    Consequently, $(\mathscr{R}\fatv)(0)=\fatv$ for all $\fatv\in H^1(\Omega)^d$. Moreover, there exists a constant $C_6>0$ such that
    \begin{align}
        ||\mathscr{R}\fatv||_{L^2(\reals^+;H^2(\Omega)^d)}^2 + ||(\mathscr{R}\fatv)_t||_{L^2(\reals^+;L^2(\Omega)^d)}^2 \leq C_6||\fatv||_{H^1(\Omega)^d}^2 \notag
    \end{align}
    for all $\fatv\in H^1(\Omega)^d$.
\end{theorem}

\end{document}